\documentclass[12pt,a4paper,reqno]{amsart}
\usepackage{amsmath,amsthm,verbatim,amssymb,amsfonts,amscd, graphicx}
\usepackage{xcolor}
\usepackage{geometry}
\usepackage{graphicx}
\usepackage{hyperref}
\usepackage{enumitem}
\usepackage{cancel}
\usepackage[english,capitalize]{cleveref}
\definecolor{darkblue}{rgb}{0,0,0.3}
\definecolor{urlblue}{rgb}{0,0,0.7}
\hypersetup{
	colorlinks=true,
	citecolor=blue,
	linkcolor=darkblue,
	urlcolor=urlblue,
}

\newtheorem{lemma}{Lemma}
\newtheorem{question}{Question}
\newtheorem{theorem}{Theorem}
\newtheorem{corollary}{Corollary}
\theoremstyle{definition}
\newtheorem{remark}{Remark}

\renewcommand{\leq}{\leqslant}
\renewcommand{\geq}{\geqslant}

\newcommand{\RR}{\mathbb{R}}

\renewcommand{\d}{\ensuremath{\mathrm d}}

\DeclareMathOperator{\Ric}{Ric}
\DeclareMathOperator{\biRic}{biRic}
\newcommand{\D}{\nabla}
\newcommand{\p}{\partial}
\renewcommand{\H}{\mathcal{H}}

\newcommand{\metric}[2]{\langle#1,#2\rangle}

\DeclareMathOperator{\vol}{vol}
\DeclareMathOperator{\diam}{diam}
\renewcommand{\tilde}{\widetilde}
\renewcommand{\bar}{\overline}

\newcommand{\preg}{\p^{\text{reg}}}
\newcommand{\psing}{\p^{\text{sing}}}
\DeclareMathOperator{\sgn}{sgn}
\renewcommand{\epsilon}{\varepsilon}

\newcommand{\TODO}[1]{}

\usepackage[maxbibnames=99,backend=biber, sorting=nyt, doi, url=false]{biblatex}
\title[Diameter and volume (spectral) comparison and isoperimetry]{New spectral Bishop--Gromov and Bonnet--Myers theorems and applications to isoperimetry}

\author{Gioacchino Antonelli}
\address{Gioacchino Antonelli
\hfill\break Courant Institute Of Mathematical Sciences (NYU), 251 Mercer Street, 10012, New York, USA}
\email{gioacchinoantonelli3@gmail.com}

\author{Kai Xu}
\address{Kai Xu
\hfill\break Department of Mathematics, Duke University, Durham, NC 27708, USA,}
\email{kai.xu631@duke.edu}

\begin{document}
\begin{abstract}
    We show a sharp and rigid spectral generalization of the classical Bishop--Gromov volume comparison theorem: if a closed Riemannian manifold $(M,g)$ of dimension $n\geq3$ satisfies
    \[
    \lambda_1\left(-\frac{n-1}{n-2}\Delta+\Ric\right)\geq n-1,
    \]
     then $\operatorname{vol}(M)\leq\operatorname{vol}(\mathbb S^{n})$, and $\pi_1(M)$ is finite. The constant $\frac{n-1}{n-2}$ cannot be improved, and if $\vol(M)=\vol(\mathbb S^n)$ holds, then $M\cong \mathbb S^{n}$. A sharp generalization of the Bonnet--Myers theorem is also shown under the   same spectral condition.
     
     The proofs involve the use of a new unequally weighted isoperimetric problem, and unequally warped $\mu$-bubbles. As an application, in dimensions $3\leq n\leq 5$, we infer sharp results on the isoperimetric structure at infinity of complete manifolds with nonnegative Ricci curvature and uniformly positive spectral biRicci curvature. 
     
     Furthermore, the main result of this paper is applied in Mazet's recent solution of the stable Bernstein problem in $\mathbb R^6$.
\end{abstract}
\maketitle

\vspace{12pt}

\tableofcontents

\section{Introduction}

\subsection{Statement of the results}

The aim of this paper is to show a spectral generalization of both the classical Bishop--Gromov volume comparison theorem and Bonnet--Myers theorem, and derive some consequences. We stress that in this paper all manifolds $M$ are without boundary.

\begin{theorem}\label{cor:CLMSArbitraryDimension}
    Let $(M^n,g)$ be an $n$-dimensional compact smooth Riemannian manifold with $n\geq 3$, $\partial M=\emptyset$, and let $0\leq \gamma\leq\frac{n-1}{n-2}$, $\lambda>0$. We denote by $\mathrm{Ric}(x):=\inf_{v\in T_xM,\,|v|=1}\mathrm{Ric}_x(v,v)$ the smallest eigenvalue of the Ricci tensor. Assume there is a positive function $u\in C^\infty(M)$ such that
    \begin{equation}\label{eqn:EnhancedSpectralCondition}
        \gamma\Delta u\leq u\Ric-(n-1)\lambda u.
    \end{equation}
    Let $\tilde M$ be the universal cover of $M$, endowed with the pull-back metric. Then we have:
    \begin{enumerate}
        \item\label{2mainthm} A diameter bound
        \begin{equation}\label{eqn:BonnetMyers}
            \diam(\tilde M)\leq\frac{\pi}{\sqrt\lambda}\cdot\Big(\frac{\max(u)}{\min(u)}\Big)^{\frac{n-3}{n-1}\gamma},
        \end{equation}   
        in particular, $\pi_1(M)$ is finite.
        \item\label{1mainthm} A sharp volume bound
        \begin{equation}\label{eqn:VoluemControl}
            \mathrm{vol}(\tilde M)\leq\lambda^{-n/2}\mathrm{vol}(\mathbb S^n).
        \end{equation}
        Moreover, if equality holds in \eqref{eqn:VoluemControl}, then every function $u$ as in \eqref{eqn:EnhancedSpectralCondition} is constant, and $\tilde M$ is isometric to the round sphere of radius $\lambda^{-1/2}$.
    \end{enumerate}
\end{theorem}

In the shorter range $0\leq \gamma < \frac{4}{n-1}$ when $n>3$, or in the same range $0\leq \gamma\leq 2$ when $n=3$, the compactness assumption of $M$ in \cref{cor:CLMSArbitraryDimension} is not needed since it is implied by the spectral condition \eqref{eqn:EnhancedSpectralCondition}. Moreover $\mathrm{diam}(M)$ is bounded from above by a constant only depending on $n,\gamma$. The latter assertions are a consequence of the works \cite{ShenYe, KaiIntermediate, CatinoRoncoroni}. So we have the following result.

\begin{corollary}\label{Corollary}
    Let $(M^n,g)$ be complete, and let $0\leq \gamma<\frac4{n-1}$ when $n>3$, or $0\leq\gamma\leq 2$ when $n=3$. Assume that there is $\lambda>0$ and a positive $u\in C^\infty(M)$ such that
    \begin{equation}\label{eqn:EnhancedSpectralConditionv2}
        \gamma\Delta u\leq u\Ric-(n-1)\lambda u.
    \end{equation}
    Then $M$ is compact, $\pi_1(M)$ is finite, and the assertions in \cref{2mainthm}, and \cref{1mainthm} in \cref{cor:CLMSArbitraryDimension} hold. In addition, there is a constant $C$ depending only on $n,\gamma$, such that $\diam(M)\leq C\lambda^{-1/2}$. In particular, $\diam(M)$ is bounded above independently of $u$.
\end{corollary}

\begin{remark}[Smoothing the eigenfunctions]\label{rmk:smoothing}
    Observe that in \eqref{eqn:EnhancedSpectralCondition} and \eqref{eqn:EnhancedSpectralConditionv2}, the function $u$ is assumed to be smooth. On the other hand, a natural condition that one usually encounters is the positivity of the Dirichlet energy
    \begin{equation}\label{eqn:PositiveSchrodinger}
        \int_M\gamma|\D\varphi|^2+\Ric \varphi^2\geq(n-1)\lambda\int_M\varphi^2,\qquad\forall\,\varphi\in C^\infty_c(M),
    \end{equation}
    for some $\gamma\geq 0$, $\lambda>0$.
    Since $\Ric$ is only a continuous function, the first eigenfunction of $-\gamma\Delta + \mathrm{Ric}$ need not be smooth. Nevertheless, we remark that all the conclusions of \cref{cor:CLMSArbitraryDimension} and \cref{Corollary} remain true under the (a prior weaker) condition \eqref{eqn:PositiveSchrodinger}. This is done by smoothing the first eigenfunction of $-\gamma\Delta + \mathrm{Ric}$, see \cref{rmk:smoothing_proof} for the proof. 
    
    We remark that, for example, condition \eqref{eqn:PositiveSchrodinger} is verified in the presence of small negative curvature wells, or under a Kato condition on $\mathrm{Ric}$: see \cite[Proposition 1.2]{ElworthySteven}, and \cite{CarronRose}, respectively. 
\end{remark}

\smallskip

Let us comment more thoroughly on the range of $\gamma$.
\begin{enumerate}[label=(\roman*), topsep=0pt, itemsep=0ex]

    \item The range $0\leq\gamma\leq\frac{n-1}{n-2}$ in \cref{cor:CLMSArbitraryDimension} is sharp, in the sense that for $\gamma>\frac{n-1}{n-2}$ none of the results there continue to hold. See \cref{rem:Supercritical} for a counterexample when $M$ is topologically $\mathbb S^1\times \mathbb S^{n-1}$. 

    \item When $n>3$, the range $0\leq \gamma < \frac{4}{n-1}$ is optimal to get a universal (i.e., only depending on $n,\gamma$) diameter bound from \eqref{eqn:EnhancedSpectralCondition}. In fact, in every dimension $n>3$, and for every $\gamma\in\big(\frac4{n-1},\frac{n-1}{n-2}\big]$, there is a compact $M^n$ that satisfies $\lambda_1(-\gamma\Delta+\Ric)\geq(n-1)$, but has arbitrarily large diameter. See \cref{rmk:no_uniform_BM} for the construction. 
        
    \item Let us emphasize the difference between \cref{cor:CLMSArbitraryDimension} and \cref{Corollary}: when $n>3$ and $\gamma\in\big[\frac4{n-1},\frac{n-1}{n-2}\big]$, the manifold $M$ satisfying the main spectral condition \eqref{eqn:EnhancedSpectralCondition} may not be compact (counterexamples include, for example, hyperbolic spaces). However, assuming that $M$ is compact, we have the volume and diameter bound, and the finiteness of $\pi_1(M)$, as a consequence of \cref{cor:CLMSArbitraryDimension}.

    On the other hand, the following conclusion can be drawn from the proof of \cref{cor:CLMSArbitraryDimension}: if the manifold $M$ in \cref{cor:CLMSArbitraryDimension} is assumed to be only complete, but the function $u$ in \eqref{eqn:EnhancedSpectralCondition} is uniformly bounded below and above, then $M$ must be compact as well, see \cref{lem:SecondPointDiameter}. Indeed, for the known noncompact examples of manifolds satisfying \eqref{eqn:EnhancedSpectralCondition} with $n>3$, and $\gamma\in\big[\frac4{n-1},\frac{n-1}{n-2}\big]$, the function $u$ either decays at infinity (e.g., hyperbolic spaces) or grows unboundedly at infinity (e.g., the examples in \cite[Subsection 3.3]{KaiIntermediate}).
\end{enumerate}
\smallskip

When $n=3$, \cref{1mainthm} in \cref{cor:CLMSArbitraryDimension} has appeared in \cite[Theorem 5.1]{CLMS}, as one step in the proof of the stable Bernstein Theorem in $\mathbb R^5$. After this paper was made public, the main \cref{cor:CLMSArbitraryDimension} (with $n=4$) together with the strategy in \cite{CLMS} has been exploited by L. Mazet \cite{Mazet} to resolve the stable Bernstein problem in $\RR^6$. We remark that the range of $\gamma$ appearing in \cref{cor:CLMSArbitraryDimension} is crucially related to the dimension restriction for the strategies used in \cite{CLMS, Mazet}. Hence, we have the following:

\begin{theorem}[Solution of the stable Bernstein problem]\label{thm:sB}
Let $M^n\hookrightarrow \mathbb R^{n+1}$, with $2\leq n\leq 5$,
be an immersed, complete, connected, two-sided, stable minimal hypersurface. Then $M$ is a Euclidean hyperplane.
\end{theorem}
The case $n=2$ of \cref{thm:sB} was solved independently in \cite{DoCarmoPeng79, FischerColbrieSchoenCPAM, Pogorelov81}. Recently, the $n=3$ case was resolved by different methods in \cite{ChodoshLi4, ChodoshLiAlternative4, CatinoMastroliaRoncoroni}. The idea in \cite{ChodoshLiAlternative4} was pushed to $n=4$ and $n=5$ in the very recent works \cite{CLMS}, and \cite{Mazet}, respectively (all these proofs rely on \cite{SSY}). When $n\geq 7$, a counterexample to \cref{thm:sB} exists, and it can be constructed using \cite{BDG69}. As of today, the case $n=6$ remains open (see \cite{SS81} and \cite{Bellettini} for progress in this dimension).

\vspace{3pt}

Let us make some further connections and comparisons with the existing results.
\begin{enumerate}[label=(\roman*), topsep=2pt, itemsep=0ex]
    \item For the case $\gamma=0$: the volume bound in \cref{cor:CLMSArbitraryDimension} reduces to the Bishop--Gromov volume comparison theorem, and, in this case, our proof reduces to the isoperimetric comparison argument due to Bray \cite{BraySingular, BrayPhD}. The diameter bound in \cref{cor:CLMSArbitraryDimension} reduces to the Bonnet--Myers theorem, for which our proof provides (up to our knowledge) a new alternative $\mu$-bubble approach.
    
    \item \cref{2mainthm} in \cref{cor:CLMSArbitraryDimension} immediately implies the same upper bound for $\mathrm{diam}(M)$. Analogously, \cref{1mainthm} in \cref{cor:CLMSArbitraryDimension} implies the same (rigid) upper bound for $\vol(M)$.
    
    \item \cref{2mainthm} becomes $\diam(\tilde M)\leq\pi/\sqrt\lambda$ in dimension $n=3$, thus being a strict spectral generalization of the Bonnet--Myers theorem. The same diameter bound when $n=3$ has appeared with a different proof in \cite[Corollary 1]{ShenYe}.
    \item Based on a Bochner formula from \cite{LiWang}, the following vanishing theorem is proved in \cite{BourCarron}: if \eqref{eqn:EnhancedSpectralCondition} holds with $\gamma=\frac{n-1}{n-2}$ and $\lambda=0$, then we have either $b_1(M)=0$, or $M$ is isometric to a quotient of the warped product $(\mathbb R\times N,\mathrm{d}r^2+\eta(r)^2h)$, where $h$ is a metric on $N$ with $\Ric_h\geq0$. By inspecting the computations in \cite[p.107]{BourCarron} we can also obtain that: for the positive case $\lambda>0$ the only possibility is $b_1(M)=0$, and for sub-critical values $\gamma<\frac{n-1}{n-2}$ (with $\lambda=0$), the result strengthens to that either $b_1(M)=0$, or $M$ is a quotient of the direct product $(\mathbb R\times N, \mathrm{d}r^2+h)$.
    
    \item In the shorter range $0\leq \gamma\leq 1/(n-2)$, the conclusion about $\pi_1(M)$ being finite in \cref{cor:CLMSArbitraryDimension} has appeared in \cite[Theorem A(ii)]{CarronRose}. The proof given here is different, and the result in \cref{cor:CLMSArbitraryDimension} shows that that threshold $0\leq \gamma\leq 1/(n-2)$ was not sharp, as already guessed by the authors in \cite{CarronRose}. As already discussed above, the example in \cref{rem:Supercritical} shows that, in \cref{cor:CLMSArbitraryDimension}, $0\leq \gamma\leq \frac{n-1}{n-2}$ is the sharp range to get that $\pi_1(M)$ is finite.

    \item In dimension 2 (which the present paper did not cover), the volume bound follows directly from Gauss--Bonnet formula. A uniform diameter bound is proved in \cite[Theorem 1.4]{KaiSurfaces}, see also \cite{SchoenYauBlackHoles}. The parameter $\beta$ in \cite{KaiSurfaces} corresponds to $1/\gamma$, therefore, the generalized Bonnet--Myers theorem holds for $\gamma<4$ in dimension $n=2$.
    
    \item Diameter bounds in presence of spectral conditions as the one in \eqref{eqn:EnhancedSpectralCondition} have appeared in \cite[Corollary 1]{ShenYe}, and \cite[Theorem A(i)]{CarronRose}. See \cite[Theorem 1.1]{CatinoRoncoroni} for another Bonnet--Myers type theorem involving both spectral and Bakry--Emery Ricci lower bounds.

    \item After we posted our paper, \cite{JiaLi, ChuHao} has proved similar results for the case of manifolds with boundary and for Bakry-\'Emery curvature conditions, respectively.
\end{enumerate}
\smallskip

\cref{cor:CLMSArbitraryDimension} has consequences on the isoperimetric structure of spaces with nonnegative Ricci curvature and uniformly positive spectral biRicci curvature. Let us recall that, given orthonormal vectors $u,v\in T_pM$, 
\[
\mathrm{biRic}(u,v):=\mathrm{Ric}(u)+\mathrm{Ric}(v)-\mathrm{Sect}(u\wedge v).
\]
In dimension $3$, note that the biRicci curvature is half of the scalar curvature.

The biRicci curvature was first introduced and studied by Shen--Ye \cite{ShenYeDuke}. There it was shown that in ambient dimensions up to 5, stable minimal hypersurfaces in manifolds with $\mathrm{biRic}\geq n-2$ enjoy a uniform diameter bound. This was motivated by and generalizes the 3-dimensional result of Schoen--Yau \cite{SchoenYauBlackHoles}. The underlying idea of \cite{ShenYeDuke} is that the stable minimal surface in question satisfies the spectral control \eqref{eqn:PositiveSchrodinger} with $\gamma=1$, and a positive $\lambda>0$. Recently, the study of manifolds with biRicci (and more general intermediate-Ricci) curvature lower bounds has found applications in the solution of the stable Bernstein problem \cite{CLMS}, and in the study of a generalization of Geroch conjecture \cite{BrendleJohneHirsch}.

Before stating \cref{thm:isop}, recall that a Riemannian manifold $(M,g)$ has \textit{linear volume growth} if there is one $p\in M$ (and hence it is true for every $p\in M$) such that 
\[\limsup_{r\to+\infty}\frac{\vol(B_r(p))}{r}<+\infty.
\]
We define the \textit{isoperimetric profile} of $M$ by
\begin{equation}\label{eq:def_ip}
    I_M(v):=\inf\big\{|\p^*E|: E\subset\subset M,\,\vol(E)=v\big\},\qquad v\in[0,\vol(M)),
\end{equation}
where $\partial^*E$ denotes the reduced boundary of $E$. Let us denote 
\[
\mathrm{biRic}(x):=\inf_{v,w\in T_xM\,\text{orthonormal}}\mathrm{biRic}_x(v,w),
\] 
and let us say that, for constants $\gamma\geq 0,\lambda>0$, a manifold satisfies \emph{$\lambda_1(-\gamma\Delta+\biRic)\geq\lambda$ outside a compact set}, if there is a compact set $K\subset M$ such that
\[\int\gamma|\D\varphi|^2+\biRic\varphi^2\geq \lambda\int\varphi^2,\qquad\forall\varphi\in C^\infty_c(M\setminus K).\]

\begin{theorem}\label{thm:isop}
    Let $(M^n,g)$ be a smooth complete noncompact Riemannian manifold with $3\leq n\leq 5$, and let $0\leq \gamma<6-n$ (or $0\leq \gamma\leq 2$ when $n=4$). Assume that $\mathrm{Ric}\geq 0$ and $\lambda_1(-\gamma\Delta+\mathrm{biRic})\geq n-2$ outside a compact set. Assume further that $M$ has one end.\footnote{If not, then $M=\mathbb R\times X$, with $\lambda_1(-\gamma\Delta+\mathrm{Ric}_X)\geq n-2$. Thus $X$ is compact, and $\vol(X)\leq \vol(\mathbb S^{n-1})$ by \cref{Corollary}. Hence $M$ has linear volume growth, and $I_M(v)\leqslant 2\vol(X)$, with equality for all $v$ large enough.}
    Then the following hold.
    \begin{enumerate}
        \item\label{item1isop} $M$ has linear volume growth.
        \item\label{item2isop} The isoperimetric profile $I_M$ satisfies a sharp bound
        \begin{equation}\label{eqn:BoundOnProfile}
            I_M(v) \leq \vol(\mathbb S^{n-1}), \qquad \forall v>0.
        \end{equation}
        Moreover, equality of \eqref{eqn:BoundOnProfile} holds for some $v>0$ if and only if there is a bounded open set $\Omega\subset M$ such that $M\setminus\Omega$ is isometric to $\mathbb S^{n-1}\times [0,+\infty)$.
    \end{enumerate}
\end{theorem}

We make the following observations on \cref{thm:isop}. 
\begin{enumerate}[label=(\roman*), topsep=0pt, itemsep=0ex]
    \item When $n=3$, observe that $\mathrm{biRic}=\mathrm{Scal}/2$. Thus, when $n=3$ and $\gamma=0$, \cref{thm:isop} extends to the noncompact case a former result of Eichmair \cite{EichmairIsop}.
    
    \item In dimension $n=3$, the structure of manifolds with $\mathrm{Ric}\geq 0$ and $\mathrm{Scal}\geq 2$ has been recently investigated in \cite{Zhu22, MunteanuWang22, ChodoshLiStryker, ZhuEtAl}, and very recently in \cite{WangRecent, WeiXuZhang} while we were completing this manuscript.
        
    In particular, when $n=3$ and $\gamma=0$, \cref{item1isop} of \cref{thm:isop} has been obtained in \cite[Theorem 5.6]{MunteanuWang22}, and generalized (with a different proof) in \cite[Theorem 1.1]{ChodoshLiStryker}. See also the recent \cite{WangRecent} for an effective version. 
       
    On the contrary, \cref{item2isop} of \cref{thm:isop}, which can be interpreted as the sharp isoperimetric control at infinity on the space, is new even when $n=3$ and $\gamma=0$. Up to the authors' knowledge, \cref{item1isop} of \cref{thm:isop} is new when $n\in \{4,5\}$, even when $\gamma=0$.
    For a question related to dimension $n\geq 6$, see \cref{quest:n6}.
    
    \item \cref{item2isop} of \cref{thm:isop} is linked to the sharp volume growth at infinity of the manifold, see \cref{rem:volumegrowth} for further details.
\end{enumerate}
\medskip

\subsection{Strategy of the proof}\label{Strategy}

Our proof of the volume bound follows the idea of isoperimetric comparison, due to Bray \cite{BrayPhD}. In \cite{Bayle04, BrayPhD}, generalizing the seminal contribution of \cite{BavardPansu86}, the authors show that in a compact manifold $M^n$ satisfying $\Ric\geq(n-1)\lambda$ with $\lambda\in\mathbb R$, the isoperimetric profile defined in \eqref{eq:def_ip} satisfies
\begin{equation}\label{eqn:SecondOrderODE}
I''_MI_M\leq-\frac{(I'_M)^2}{n-1}-(n-1)\lambda,
\end{equation}
in the viscosity sense. When $\lambda>0$, an ODE comparison can be applied to \eqref{eqn:SecondOrderODE} to obtain a sharp upper bound on the existence interval of $I_M$, and therefore an upper bound on $\vol(M)$. In \cite{CLMS} this argument is generalized to the spectral case in dimension 3: given a function $u$ satisfying \eqref{eqn:EnhancedSpectralCondition}, one considers the weighted isoperimetric profile
\[
I(v):=\inf\left\{\int_{\p^*E}u^\gamma: E\subset\subset M,\,\int_E u^\gamma=v\right\}.
\]
The same sharp inequality $I''I\leq-\frac12(I')^2-2\lambda$ is shown in \cite{CLMS}, which eventually leads to a sharp volume bound. In other dimensions, however, the inequality obtained through this argument requires the condition $\gamma<\frac4{n-1}$ (which is strictly narrower than $\gamma\leq\frac{n-1}{n-2}$ when $n>3$), and the volume bound is not sharp.

The main novelty of our approach is to consider the unequally weighted isoperimetric profile
\begin{equation}\label{eqn:UnequallyWighted}
I(v):=\inf\left\{\int_{\p^*E}u^\gamma: E\subset\subset M, \int_E u^{\frac{2\gamma}{n-1}}=v\right\},
\end{equation}
which, if \eqref{eqn:EnhancedSpectralCondition} is in force, satisfies the sharp inequality
\begin{equation}\label{eqn:viscosity_aux}
    I''I\leq-\frac{(I')^2}{n-1}-(n-1)\lambda.
\end{equation}
This leads to the proof of the sharp volume bound \eqref{eqn:VoluemControl}.
\vspace{3pt}

The diameter bound is obtained by employing the idea of $\mu$-bubbles originally due to Gromov \cite{Gromov_FourLectures}. In \cite{KaiIntermediate}, the second-named author considered a warped $\mu$-bubble obtained by minimizing a (suitably defined) energy
\[E(\Omega)=\int_{\p^*\Omega}u^\gamma-\int_\Omega hu^\gamma,\]
where $u$ is the function satisfying \eqref{eqn:EnhancedSpectralCondition}, and $h$ satisfies a certain differential inequality
\[|\D h|\leq C(n,\gamma)h^2+(n-1)\lambda.\]
Within the narrower range $0\leq \gamma<\frac4{n-1}$, this argument leads to the diameter bound appearing in \cref{Corollary}. We refer to \cite[Theorem 1.10]{KaiIntermediate} for more details.

In our case, to obtain \cref{2mainthm} of \cref{cor:CLMSArbitraryDimension} for the full range $0\leq\gamma\leq\frac{n-1}{n-2}$, we consider an unequally warped $\mu$-bubble defined by the functional
\[
E(\Omega)=\int_{\p^*\Omega}u^\gamma-\int_\Omega hu^{\frac{2\gamma}{n-1}},
\]
where now it turns out that $h$ needs to be chosen carefully so that
\[|\D h|u^{\gamma\frac{3-n}{n-1}}<(n-1)\lambda+\frac{h^2}{n-1}u^{2\gamma\frac{3-n}{n-1}}.\]
Under the situation of \cref{cor:CLMSArbitraryDimension}, especially since $u$ is bounded below and above uniformly, the latter argument leads to the diameter bound of $\tilde M$ stated in \eqref{eqn:BonnetMyers}. In particular, this implies $\widetilde M$ is compact and $\pi_1(M)$ is finite.
\smallskip

We remark that \cref{cor:CLMSArbitraryDimension} holds in all dimensions. Often in the literature, results that are proved through generalized minimal surfaces are stated up to dimension $n\leq7$, because of singularity formation in higher dimensions. In Appendix \ref{sec:Case8}, we give a detailed argument showing that our proofs extend to the singular case $n\geq8$ as well. The core content is to carry out the stability inequality for variations supported away from the singular set. We argued along the lines of similar approaches that appeared in \cite{Bayle04, BraySingular}.
\medskip

Finally, the proof of \cref{thm:isop} exploits \cref{cor:CLMSArbitraryDimension} and results and idea presented in \cite{ChodoshLiStryker, XingyuZhu, KaiIntermediate, CLMS}. Roughly speaking, within the range of dimension $3\leq n\leq 5$, and when $0\leq \gamma<6-n$ (or $0\leq \gamma\leq 2$ when $n=4$), the assumption $\lambda_1(-\gamma\Delta+\mathrm{biRic})\geq n-2$ implies that for every $\varepsilon>0$, one can find a stable warped $\mu$-bubble $\Omega$ such that
\begin{equation}\label{eqn:ControlMuBubble}
B_{\varepsilon^{-1}}(o)\subset \Omega, \qquad -\frac{4}{4-\gamma}\Delta_{\partial\Omega}+\mathrm{Ric}_{\partial\Omega}- (n-2-\varepsilon)\geq 0.
\end{equation}
Since $\mathrm{Ric}\geq 0$ (thus $b_1(M)<+\infty$) and $M$ has one end, the construction can be made such that $\partial\Omega$ is connected. Hence  \eqref{eqn:ControlMuBubble}, together with the volume bound in \cref{cor:CLMSArbitraryDimension}, implies that $|\p\Omega|\leq\vol(\mathbb S^{n-1})+o_\varepsilon(1)$. The conclusion about $\Omega$, coupled with the fact that $I_M$ is nondecreasing (since $\Ric\geq0$), essentially implies \eqref{eqn:BoundOnProfile}.

In higher dimensions $n\geq6$, the present $\mu$-bubble argument does not provide hypersurfaces satisfying the spectral condition \eqref{eqn:ControlMuBubble}. Thus, we conclude the introduction by asking the following question about \cref{thm:isop}.

\begin{question}\label{quest:n6}
    Let $n\geq 6$ be a natural number. Is it possible to construct a smooth complete noncompact Riemannian manifold $(M^n,g)$ such that $\mathrm{Ric}\geq 0$ everywhere, $\mathrm{biRic}\geq n-2$ outside a compact set, and 
    \begin{itemize}
        \item either $M$ doesn't have linear volume growth;
        \item or $\lim_{v\to+\infty}I_M(v)>\vol(\mathbb S^{n-1})$?
    \end{itemize}
\end{question}

In the very recent paper \cite{ZhuZhou} (appeared after the present paper was made public) the authors give a negative answer to \cref{quest:n6}.

\addtocontents{toc}{\protect\setcounter{tocdepth}{0}}

\subsection{Organization of the paper} 
In \cref{sec:bonnet_myers} we prove \cref{2mainthm} of \cref{cor:CLMSArbitraryDimension}. The main tool is \cref{lem:SecondPointDiameter}, which uses unequally warped $\mu$-bubbles. Then in \cref{rmk:no_uniform_BM}, we discuss the possibility of improving \eqref{eqn:BonnetMyers} in the range $\gamma\in\left[\frac{4}{n-1},\frac{n-1}{n-2}\right]$. In \cref{sec:VolumeBounds} we prove \cref{1mainthm} of \cref{cor:CLMSArbitraryDimension}. First, in \cref{lem:Viscosity} we derive the viscosity ODE \eqref{eqn:viscosity_aux} for the isoperimetric profile \eqref{eqn:UnequallyWighted}. Then the main result is proved by ODE comparison (see \cref{lem:AnotherComparison}). We also discuss the possibility of direct isoperimetry comparison on $\tilde M$ (see \cref{rmk:direct_isop}), the supercritical case $\gamma>\frac{n-1}{n-2}$ (see \cref{rem:Supercritical}), and we prove \cref{rmk:smoothing} in \cref{rmk:smoothing_proof}.

Then in \cref{sec:isop} we prove \cref{thm:isop}. We first produce large sets with small boundary areas by means of \cref{existencegoodmububble}, and then we prove the main ancillary \cref{lem:CentralLemmaForIsop}, from which the proof of \cref{thm:isop} follows. We also discuss the implications of the results in \cref{thm:isop} to sharp linear volume growth of the manifold at infinity in \cref{rem:volumegrowth}.

Finally, in \cref{sec:Case8} we present the proofs in the singular case $n\geq 8$.

\subsection{Notations}
All the manifolds in this paper are assumed to be smooth, complete, and without boundary. We will often omit the volume  measure $\mathrm{d}\mathrm{vol}$ and the area measure $\mathrm{d}\H^{n-1}$ in the integrals. We say that $A\subset\subset B$ if the closure of $A$ is relatively compact in $B$. For the theory of finite perimeter sets, we refer the reader to \cite{MaggiBook}. For a set $E$ of locally finite perimeter, we use $\p^*E$ to denote the reduced boundary of $E$. In dimensions 7 or below, all the minimizers obtained in this paper through Geometric Measure Theory techniques are actually smooth, so the reader may assume $\p^*E=\p E$.  

\subsection{Acknowledgments} 
G.A. acknowledges the financial support of the Courant Institute, the AMS-Simons Travel grant, and the NSF DMS Grant No. 2505713. K.X. would like to thank the hospitality of the Courant Institute, where part of this work was done during his visit. The authors are grateful to Chao Li, Xian-Tao Huang, Laurent Mazet, Luciano Mari, Marco Pozzetta, and Alberto Roncoroni for interesting discussions and suggestions on the topic of the paper. The authors are grateful to Otis Chodosh for comments that led to the improvement of a previous version of \cref{thm:isop}, and to Jintian Zhu for comments that led to \cref{rmk:no_uniform_BM}. 

The authors discussed \cref{thm:isop}, which motivated the present paper, when they were participating in the workshop ``Recent Advances in Comparison Geometry'' in Hangzhou. They thank the Banff International Research Station and the Institute of Advanced Studies in Mathematics for this enriching opportunity.

\addtocontents{toc}{\protect\setcounter{tocdepth}{2}}

\section{Unequally warped \texorpdfstring{$\mu$}{μ}-bubbles and diameter bounds}\label{sec:bonnet_myers}

Let $n\geq 3$, and let $(M^n,g)$ be a complete Riemannian manifold, not necessarily compact. Let $u\in C^\infty(M)$ be a positive function. Given a parameter $0\leq \gamma\leq \frac{n-1}{n-2}$, we fix 
$$
\alpha:=\frac{2\gamma}{n-1}.
$$ 
Assume that $u\in C^\infty(M)$ is positive and such that 
\begin{equation}\label{eqn:MainSpectral}
-\gamma\Delta u+\Ric u\geq(n-1)\lambda u.
\end{equation}
Let $\Omega_-\subset\Omega_+\subset M$ be two domains with nonempty boundaries, such that $\bar{\Omega_+}\setminus\Omega_-$ is compact. Suppose $h\in C^\infty(\Omega_+\setminus\bar{\Omega_-})$ satisfies
\begin{equation}\label{eqn:condition_for_h}
    \lim_{x\to\p\Omega_-}h(x)=+\infty,\qquad
    \lim_{x\to\p\Omega_+}h(x)=-\infty\qquad\text{uniformly.}
\end{equation}
For an arbitrary fixed domain $\Omega_0$ with $\Omega_-\subset\subset\Omega_0\subset\subset\Omega_+$, consider the functional
\begin{equation}\label{eqn:MuBubbleFunctional}
    E(\Omega):=\int_{\p^*\Omega}u^\gamma-\int(\chi_\Omega-\chi_{\Omega_0})hu^\alpha,
\end{equation}
defined on sets of finite perimeter, where $\p^*\Omega$ is the reduced boundary of $\Omega$.
By a slight variant of \cite[Proposition 2.1]{JintianZhu} there always exists a set of finite perimeter $\Omega$, with $\Omega_-\subset\subset\Omega\subset\subset\Omega_+$, which is a minimizer of the energy $E$. Assume, for the purpose of the following computations, that $\Omega$ has smooth boundary.

Let $\nu$ denote the outer unit normal at $\p\Omega$, and let $\varphi\in C^\infty(M)$. Let $H,\mathrm{II}$ be the mean curvature and second fundamental form of $\partial\Omega$, computed with respect to $\nu$. For a function $f\in C^\infty(M)$ we denote $f_\nu:=\langle \nabla f,\nu\rangle$. For an arbitrary smooth variation $\{\Omega_t\}_{t\in(-\varepsilon,\varepsilon)}$ with $\Omega_0=\Omega$ and variational field $\varphi\nu$ at $t=0$, we compute the first variation
\begin{equation}\label{eq:mububble_H}
    0=\frac{\mathrm{d}}{\mathrm{d}t}E(\Omega_t)\Big|_{t=0}=\int_{\p\Omega}\big(H+\gamma u^{-1}u_\nu-hu^{\alpha-\gamma}\big)u^\gamma\varphi.
\end{equation}
Since $\varphi$ is arbitrary we have $H=hu^{\alpha-\gamma}-\gamma u^{-1}u_\nu$. Then, computing the second variation,
\[\begin{aligned}
    0 &\leq \frac {\mathrm{d}^2}{\mathrm{d}t^2}E(\Omega_t)\Big|_{t=0}
    = \int_{\p\Omega}\Big[-\Delta_{\p\Omega}\varphi-|\text{II}|^2\varphi-\Ric(\nu,\nu)\varphi-\gamma u^{-2}u_\nu^2\varphi \\
    &\hspace{130pt} +\gamma u^{-1}\varphi\big(\Delta u-\Delta_{\p\Omega}u-Hu_\nu\big)-\gamma u^{-1}\metric{\D_{\p\Omega}u}{\D_{\p\Omega}\varphi} \\
    &\hspace{130pt} -h_\nu u^{\alpha-\gamma}\varphi+(\gamma-\alpha)hu^{\alpha-\gamma-1}u_\nu\varphi\Big]u^\gamma\varphi.
\end{aligned}\]
Setting $\varphi=u^{-\gamma}$, using our main assumption \eqref{eqn:MainSpectral}, and integrating by parts, we have
\[\begin{aligned}
    0 &\leq \int_{\p\Omega} -|\text{II}|^2u^{-\gamma}-\Ric(\nu,\nu)u^{-\gamma}-\gamma u^{-2-\gamma}u_\nu^2+\gamma u^{-1-\gamma}\Delta u-\gamma u^{-1-\gamma}\Delta_{\p\Omega}u \\
    &\qquad\qquad -\gamma Hu^{-1-\gamma}u_\nu-\gamma u^{-1}\metric{\D_{\p\Omega}u}{\D_{\p\Omega}(u^{-\gamma})}-h_\nu u^{\alpha-2\gamma}+(\gamma-\alpha)hu^{\alpha-2\gamma-1}u_\nu \\
    &\leq \int_{\p\Omega}u^{-\gamma}\Big[-\frac{H^2}{n-1}-(n-1)\lambda-\gamma(u^{-1}u_\nu)^2 \\
    &\hspace{65pt} -\gamma H(u^{-1}u_\nu)+|\D h|u^{\alpha-\gamma}+(\gamma-\alpha)hu^{\alpha-\gamma}(u^{-1}u_\nu)\Big].
\end{aligned}\]
Setting $X=hu^{\alpha-\gamma}$ and $Y=u^{-1}u_\nu$ (so $H=X-\gamma Y$), we have
\begin{align}
    &\begin{aligned}
    0 &\leq \int_{\p\Omega} u^{-\gamma}\Big[-\frac{X^2}{n-1}+\frac{2\gamma}{n-1}XY-\frac{\gamma^2}{n-1}Y^2-(n-1)\lambda \\
    &\qquad\qquad -\gamma Y^2-\gamma(X-\gamma Y)Y+|\D h|u^{\alpha-\gamma}+(\gamma-\alpha)XY\Big]
    \end{aligned} \label{eq:mububbleineq}\\
    &\ \leq \int_{\p\Omega}u^{-\gamma}\Big[
        -\frac{X^2}{n-1}-(n-1)\lambda+|\D h|u^{\alpha-\gamma}
        +\Big(\frac{2\gamma}{n-1}-\alpha\Big)XY
        +\Big(\frac{n-2}{n-1}\gamma^2-\gamma\Big)Y^2\Big]. \nonumber
\end{align}
Recall that $\alpha=\frac{2\gamma}{n-1}$ and $0\leq\gamma\leq\frac{n-1}{n-2}$. Therefore, if $h$ satisfies
\begin{equation}\label{eqn:h_must_satisfy}
    |\D h|u^{\alpha-\gamma}<\frac{h^2u^{2\alpha-2\gamma}}{n-1}+(n-1)\lambda,
\end{equation}
then we obtain a contradiction. We are thus ready to prove the following.
\begin{lemma}\label{lem:SecondPointDiameter}
    Let $(M^n,g)$ be a complete Riemannian manifold with $n\geq 3$, and let $0\leq \gamma\leq \frac{n-1}{n-2}$, $\lambda>0$. Assume there is $u\in C^\infty(M)$ such that $\inf(u)>0$, $\sup(u)<+\infty$, and  
    \[-\gamma\Delta u+\Ric u\geq(n-1)\lambda u.\]
    Then
    \begin{equation}\label{eqn:BonnetMyersv2}
        \diam(M)\leq\frac{\pi}{\sqrt\lambda}\cdot\Big(\frac{\sup(u)}{\inf(u)}\Big)^{\frac{n-3}{n-1}\gamma}.
    \end{equation}
    In particular, $M$ must be compact.
\end{lemma}
\begin{proof}[Proof $(n\leq7)$]
    For the ease of readability, we give the proof assuming $n\leq 7$. The case $n\geq 8$, where one needs to deal with the possible singularity of minimizers, is postponed to \cref{sec:Case8}. Set $\alpha:=\frac{2\gamma}{n-1}$ as usual (note that $\alpha-\gamma\leq0$). First, notice that if $h$ is a smooth function on $M$ such that
    \begin{equation}\label{eqn:EQNONH}
        |\nabla h|<\frac{\sup(u)^{2\alpha-2\gamma}}{(n-1)\inf(u)^{\alpha-\gamma}}h^2+\frac{(n-1)\lambda}{\inf(u)^{\alpha-\gamma}}=:Ch^2+D,
    \end{equation}
    then 
    \begin{equation}\label{eqn:AbsurdMuBubble}
    \begin{aligned}
        |\nabla h|u^{\alpha-\gamma} &\leq |\nabla h|\inf(u)^{\alpha-\gamma} < \frac{h^2u^{2\alpha-2\gamma}}{n-1}+(n-1)\lambda.\\
    \end{aligned}
    \end{equation}
    Moreover, notice that
    \begin{equation}\label{eqn:CD}
        CD=\lambda\cdot \left(\frac{\sup(u)}{\inf(u)}\right)^{2\frac{3-n}{n-1}\gamma}.
    \end{equation}
    Suppose by contradiction \eqref{eqn:BonnetMyersv2} is false. Then there is $\varepsilon>0$ such that
    \begin{equation}\label{eqn:BonnetMyersv2Not}
        \diam(M)>\frac{\pi}{\sqrt\lambda}\cdot\Big(\frac{\sup(u)}{\inf(u)}\Big)^{\frac{n-3}{n-1}\gamma}\cdot(1+\varepsilon)^2+2\varepsilon.
    \end{equation}
    Let us now fix a point $p\in M$ realizing $\mathrm{diam}(M)$. Take $\Omega_-:=B_\varepsilon(p)$, and let $d:M\setminus\Omega_-\to\mathbb R$ be a smoothing of $d(\cdot,\p\Omega_-)$ such that
    \[d|_{\p\Omega_-}\equiv0, \quad
        |\nabla d|\leq 1+\varepsilon,\quad
        d\geq\frac{d(\cdot,\p\Omega_-)}{1+\varepsilon}.\]
    It can be verified by direct computations that \begin{equation}\label{eqn:definitionh}
    h(x):=\sqrt{\frac{D}{C}}\cot\left(\frac{\sqrt{CD}}{1+\varepsilon}d(x)\right)
    \end{equation}
    satisfies \eqref{eqn:EQNONH}, and thus \eqref{eqn:AbsurdMuBubble}. Now we have that
    \begin{equation}\label{eqn:TheSetO}
        \mathcal{O}:=\Big\{d>\frac{(1+\varepsilon)\pi}{\sqrt{CD}}\Big\}\supset\Big\{d(\cdot,p)>\varepsilon+\frac{(1+\varepsilon)^2\pi}{\sqrt{CD}}\Big\}\ne\emptyset
    \end{equation}
    due to \eqref{eqn:BonnetMyersv2Not}. Set $\Omega_+:=M\setminus\overline{\mathcal{O}}$. Thus we have found two domains $\Omega_-\subset\subset\Omega_+\subset\subset M$ and $h(x)\in C^\infty(\Omega_+\setminus\bar{\Omega_-})$ which satisfy \eqref{eqn:condition_for_h}. Let $\Omega$ be an unequally warped $\mu$-bubble related to the functional \eqref{eqn:MuBubbleFunctional}. Since $n\leq 7$, by classical results in Geometric Measure Theory (see \cite[Theorem 27.5]{MaggiBook}, and \cite[Theorem 28.1]{MaggiBook}) we have that $\Omega$ has smooth boundary. Since \eqref{eqn:AbsurdMuBubble} (hence \eqref{eqn:h_must_satisfy}) is in force, we get a contradiction with the minimizing property of $E$, as shown in the computations of the second variation of $E$ right before the Lemma.
\end{proof}

\begin{proof}[Proof of \cref{2mainthm} of \cref{cor:CLMSArbitraryDimension}]
    Let $\pi:\tilde M\to M$ be the universal cover of $M$. Set $\tilde g=\pi^*g$ and $\tilde u:=u\circ\pi$, thus we have
    \[
    \gamma \tilde{\Delta} \tilde u \leq \tilde u \widetilde{\mathrm{Ric}}-(n-1)\lambda\tilde u.
    \]
    The assertion directly comes from \cref{lem:SecondPointDiameter}.
\end{proof}

\begin{remark}[No universal diameter bound]\label{rmk:no_uniform_BM}     Suppose $n\geq4$ and $\frac4{n-1}<\gamma\leq\frac{n-1}{n-2}$, which is the range included in \cref{cor:CLMSArbitraryDimension} but not in \cref{Corollary}. Thus, we only obtain the diameter bound \eqref{eqn:BonnetMyers} depending on the maximum and minimum of $u$. Here let us show that, even assuming $M$ is closed, there is no universal diameter upper bound on $M$ (that depends only on $n,\gamma$).
    
    Fix $L\gg1$. Fix a cutoff function $\eta_0$ so that $\eta_0|_{[-\infty,0)}\equiv0$ and $\eta_0|_{[1,\infty)}\equiv1$. Let $\delta,\mu>0$ be constants to be chosen. Let $\eta$ be the cutoff function such that:

    \begin{enumerate}[label=(\arabic*), topsep=0pt, itemsep=0.2ex]
        \item $\eta\equiv0$ on $(-\infty,-L-\mu]\cup[-\delta,\delta]\cup[L+\mu,+\infty)$ and $\eta\equiv1$ on $[-L,-2\delta]\cup[2\delta,L]$,
        \item on each interval $[-L-\mu,-L]$, $[-2\delta,-\delta]$, $[\delta,2\delta]$, $[L,L+\mu]$, the function $\eta$ interpolates between 0 and 1 according to the model function $\eta_0$.
    \end{enumerate}
    Set $\alpha=\frac{2\gamma}{n-1}$ as usual. Choose constants $a,b>0$ such that
    \begin{equation}\label{eq:const_ab}
        -\frac{a^2}{n-1}-\Big(\gamma-\frac{n-2}{n-1}\gamma^2\Big)\frac{b^2}{(\gamma-\alpha)^2}-(n-1)+ab=0.
    \end{equation}
    The existence of such numbers follows by calculating the discriminant:
    \[\begin{aligned}
        \exists\,a,b>0\text{ satisfying \eqref{eq:const_ab}}\ \ &\Leftrightarrow\ \ \frac{4}{n-1}\cdot\Big(\gamma-\frac{n-2}{n-1}\gamma^2\Big)\frac{1}{(\gamma-\alpha)^2}<1 \\
        &\Leftrightarrow\ \ \gamma>\frac4{n-1},
    \end{aligned}\]
    and the second line is exactly our assumption. From the sign of the terms in \eqref{eq:const_ab}, it follows that we can choose $a,b$ to be positive.
    
    Now we let $u(r)$ solve the differential equation
    \begin{equation}\label{eq:ODEu}
        \left\{\begin{aligned}
        & u'(r) = \eta(r)\sgn(r)\cdot\frac{b}{\gamma-\alpha}u(r),\\
        & u(0)=1.
    \end{aligned}\right.
    \end{equation}
    The solution is positive, even, and smooth on the entire $\RR$. Also, note that $u\equiv1$ in $[-\delta,\delta]$, and $u$ is constant on $(-\infty,-L-\mu]\cup[L+\mu,+\infty)$. Then let $h=h(r)$ solve
    \begin{equation}\label{eq:ODEh}
        \left\{\begin{aligned}
            & h'u^{\alpha-\gamma}=-\frac{h^2u^{2\alpha-2\gamma}}{n-1}-\Big(\gamma-\frac{n-2}{n-1}\gamma^2\Big)\frac{(u')^2}{u^2}-(n-1), \\
            & h(0)=0.
        \end{aligned}\right.
    \end{equation}
    Since the solution is an odd function, below we only discuss within the range $r\geq0$ (and the case $r<0$ follows by symmetry). We claim that there is a choice of $\delta$, such that the solution of \eqref{eq:ODEu} \eqref{eq:ODEh} satisfies $h(2\delta)u(2\delta)^{\alpha-\gamma}=-a$. To see this, we set a new function $Q=hu^{\alpha-\gamma}$, and thus \eqref{eq:ODEh} is converted to
    \begin{equation}\label{eq:ODE_for_Q}
        \left\{\begin{aligned}
            & Q'=-\frac{Q^2}{n-1}-\Big(\gamma-\frac{n-2}{n-1}\gamma^2\Big)\frac{\eta^2b^2}{(\gamma-\alpha)^2}-(n-1)-\eta bQ, \\
            & Q(0)=0.
        \end{aligned}\right.
    \end{equation}
    When $\delta\to0$ the solution satisfies $Q(2\delta)\to0$. On the other hand, note that $\eta=0$ on $[0,\delta]$, thus in this interval we have $Q'=-\frac{Q^2}{n-1}-(n-1)$. Thus there exists $\delta_0>0$ such that, when $\delta\nearrow\delta_0$ we have $Q(2\delta)\to-\infty$. By continuity, there exists a $\delta$ as claimed.

    Next, note that $Q\equiv-a$ is a solution of \eqref{eq:ODE_for_Q} in $[2\delta,L]$, due to our choice \eqref{eq:const_ab} and $\eta\equiv1$. By uniqueness of ODE, for our specific choice of $\delta$ the solution of \eqref{eq:ODE_for_Q} must satisfy $Q\equiv-a$ on $[2\delta,L]$. Thus for \eqref{eq:ODEh}, we have $h=-au^{\gamma-\alpha}$ on $[2\delta,L]$.

    Next, by continuity, there is a sufficiently small $\mu$ such that the solution of \eqref{eq:ODEh} exists on $[-L-\mu,L+\mu]$. When $r\geq L+\mu$, \eqref{eq:ODEh} becomes
    \[
    h'u(L+\mu)^{\alpha-\gamma}=-\frac{u(L+\mu)^{2\alpha-2\gamma}}{n-1}h^2-(n-1),
    \]
    thus the solution is
    \[
    h(r)=-(n-1)u(L+\mu)^{-(\alpha-\gamma)}\cot(r_0-r),
    \]
    for some $r_0>L+\mu$. To summarize, the equation \eqref{eq:ODEh} has a solution $h=h(r)$ on a maximal interval $(-r_0,r_0)\supset[-L,L]$, and $h$ blows up like $-(n-1)u(L+\mu)^{-(\alpha-\gamma)}\cot|r_0\mp r|$ near the endpoints $\pm r_0$.
    
    Finally, we let $f(r)$ solve the ODE
    \begin{equation}\label{eqn:ODEf}
        \left\{\begin{aligned}
            & (n-1)\frac{f'}{f}=hu^{\alpha-\gamma}-\gamma u^{-1}u', \\
            & f(0)=1.
        \end{aligned}\right.
    \end{equation}
    This equation, along with \eqref{eq:ODEh}, come from analyzing the equality case in the computations from \eqref{eq:mububble_H} to \eqref{eqn:h_must_satisfy}. Note that $f$ is even, since $u$ is even and $h$ is odd. When $r\geq L+\mu$, the ODE becomes
    \[(n-1)\frac{f'(r)}{f(r)}=-(n-1)\cot(r_0-r)\ \ \Rightarrow\ \ f(r)=c\sin(r_0-r)\ \ \text{for some $c>0$.}\]
    For $\epsilon\ll1$ consider the metric $g=\mathrm{d} r^2+\epsilon^2f(r)^2g_{\mathbb S^{n-1}}$. By the asymptotic of $f$ near $\pm r_0$, this represents a metric on $\mathbb S^{n}$ with acute cone singularities at the two poles $r=\pm r_0$. Note that $\Ric(\p_r,\p_r)=-(n-1)f''/f$, and $\Ric(e,e)$ (where $e$ is tangent to $\mathbb S^{n-1}$) grows like $r^{-2}$ near the poles. So we may choose $\varepsilon$ so small such that $\Ric(\p_r,\p_r)$ is the minimal eigenvalue of $\Ric$ (see, e.g., the argument at the beginning of \cref{rem:Supercritical}).
    
    From the ODEs \eqref{eq:ODEh}, \eqref{eqn:ODEf} it can be directly verified that
    \[
    -\gamma\Big[\frac{u''}{u}+(n-1)\frac{f'}{f}\frac{u'}{u}\Big]-(n-1)\frac{f''}f=(n-1).
    \]
    This implies $-\gamma\Delta u + u\Ric(\p_r,\p_r) = (n-1)u$.

    Since $g$ is conic and $u$ is constant near the pole, we can smooth the metric near the pole, while not decreasing the minimal eigenvalue of $\Ric$. As a consequence, we still have $-\gamma\Delta u+\Ric u\geq(n-1)u$ in the smoothed metric. This implies $\lambda_1(-\gamma\Delta+\Ric)\geq n-1$. On the other hand, note that $\diam(M,g)>2L$, which can be made arbitrarily large.
\end{remark}

\section{Unequally weighted isoperimetric profile and volume bounds}\label{sec:VolumeBounds}

Let $n\geq 3$, and $(M^n,g)$ be a complete Riemannian manifold. Let $0\leq \gamma\leq \frac{n-1}{n-2}$, and set
$$
\alpha:=\frac{2\gamma}{n-1}.
$$ Set the weighted volume $V_0:=\int_M u^\alpha\in(0,\infty]$, and define the unequally weighted isoperimetric profile
\begin{equation}\label{eq:weighted_ip}
    I(v):=\inf\left\{\int_{\p^* E}u^\gamma:\ E\subset\subset M\ \text{has finite perimeter, and}\ \int_E u^\alpha=v\right\},
\end{equation}
for all $v\in[0,V_0)$. Here $\partial^* E$ denotes the reduced boundary of $E$.

\begin{lemma}\label{lem:Viscosity}
    Let $\gamma,\alpha,V_0,I(v)$ be as above, and let $\lambda>0$. Suppose $M^n$ is complete. Assume $u\in C^\infty(M)$ satisfies $\inf(u)=1$ and
    \begin{equation}\label{eq:aux1}
    \gamma\Delta u\leq u\Ric-(n-1)\lambda u.
    \end{equation}
    Suppose for a fixed $v_0\in(0,V_0)$, there exists a bounded set $E$ with finite perimeter, such that $\int_E u^\alpha=v_0$ and $\int_{\p^* E}u^\gamma=I(v_0)$. Then $I$ satisfies
    \begin{equation}\label{eq:viscosity_aux}
        I''I\leq-\frac{(I')^2}{n-1}-(n-1)\lambda
    \end{equation}
    in the viscosity sense at $v_0$.
\end{lemma}
\begin{proof}[Proof $(n\leq7)$]
    For the ease of readability, we present the proof here in the case $n\leq 7$. When $n\geq 8$ one again needs to deal with the possible singularity of minimizers, and we postpone the resolution of this to \cref{sec:Case8}.

    We recall that \eqref{eq:viscosity_aux} holds in the viscosity sense at $v_0$ if the following holds: for every smooth function $\varphi$ defined in some neighborhood $\mathcal{I}_{v_0}\ni v_0$, such that $\varphi(v_0)=I(v_0)$ and $\varphi(x)\leq I(x)$ in $\mathcal{I}_{v_0}$, we have
    \begin{equation}\label{eqn:ViscosityToShow}
    -\varphi''(v_0)\varphi(v_0)\geq (n-1)\lambda+\frac{\varphi'(v_0)^2}{n-1}.
    \end{equation}
    
    The viscosity inequality follows by finding an upper barrier of the isoperimetric profile satisfying \eqref{eqn:ViscosityToShow}, and the latter is done by computing the second variation of the weighted area. Notice that by classical results in Geometric Measure Theory (see, e.g., \cite[Section 3.10]{MorganRegular}), the set $E$ in the statement has smooth boundary, since $n\leq 7$. For a bounded set $F\subset M$ with smooth boundary, define the weighted volume and perimeter
    \[V(F):=\int_Fu^\alpha,\qquad A(F):=\int_{\p F}u^\gamma.\]
    
    For $\varphi\in C^\infty(M)$, let us consider a smooth family of sets $\{E_t\}_{t\in(-\varepsilon,\varepsilon)}$, such that $E_0=E$, the variational vector field $X_t$ along $\p E_t$ at $t=0$ is $\varphi\nu$ (where $\nu$ denotes the outer unit normal of $\p E_t$), and $\nabla_{X_t}X_t = (\varphi\varphi_\nu)\nu$ at $t=0$. Define $V(t):=V(E_t)$ and $A(t):=A(E_t)$; thus $V(0)=v_0$, $A(0)=I(v_0)$. Let $H$ be the mean curvature of $\partial E$, and denote $\varphi_\nu:=\metric{\D\varphi}{\nu}$, $u_\nu:=\metric{\D u}{\nu}$ for brevity. We compute the first and second variations of the (weighted) volume at the initial time:
    \begin{equation}\label{eqn:VariationsVolume}
        \frac{\d V}{\d t}(0)=\int_{\p E} u^\alpha\varphi,\qquad \frac{\d ^2V}{\d t^2}(0)=\int_{\p E}(H+\alpha u^{-1}u_\nu)u^\alpha\varphi^2+u^\alpha\varphi\varphi_\nu,
    \end{equation}
    where the term $u^\alpha\varphi\varphi_\nu$ in the last equality appears due to our assumption have that the acceleration $\nabla_{X_t}X_t$ is $(\varphi\varphi_\nu)\nu$. Next, compute the first variation of the (weighted) area:
    \begin{equation}\label{eqn:FirstVariationOfArea}
        \frac {\d A}{\d t}(0)=\int_{\p E} u^\gamma\varphi(H+\gamma u^{-1}u_\nu)=\int_{\p E} u^\alpha\varphi\cdot u^{\gamma-\alpha}(H+\gamma u^{-1}u_\nu).
    \end{equation}
    Finally, the second variation of the area is
    \begin{equation}\label{eqn:SecondVariation}
        \begin{aligned}
            \frac {\d^2 A}{\d t^2}(0) &= \int_{\p E} \Big(-\Delta_{\p E}\varphi-\Ric(\nu,\nu)\varphi-|\mathrm{II}|^2\varphi\Big)u^\gamma\varphi \\
        &\qquad\qquad +\Big(-\gamma u^{-2}u_\nu^2\varphi+\gamma u^{-1}\D^2u(\nu,\nu)\varphi-\gamma u^{-1}\metric{\D_{\p E} u}{\D_{\p E}\varphi}\Big)u^\gamma\varphi \\
        &\qquad\qquad +\left(\gamma u^{\alpha-1}u_\nu\varphi^2+u^{\alpha}\varphi\varphi_\nu +Hu^{\alpha}\varphi^2\right) u^{\gamma-\alpha}(H+\gamma u^{-1}u_\nu),
        \end{aligned}
    \end{equation}
    where the term $u^\alpha\varphi\varphi_\nu$ in the last line again appears from evaluating $\D_{X_t}X_t$, as noted above. Since $E$ is a volume-constrained minimizer, we must have $\frac{\mathrm{d}A}{\mathrm{d}t}(0)=0$ whenever $\frac{\mathrm{d}V}{\mathrm{d}t}(0)=0$. Hence $u^{\gamma-\alpha}(H+\gamma u^{-1}u_\nu)$ is constant due to \eqref{eqn:FirstVariationOfArea}.
    
    From now on, we fix the choice $\varphi=u^{-\gamma}$ in the variation. Since $V(t)$ is strictly monotone in $t$ in a neighbourhood of 0, we may view $A$ as a smooth function in $V$, defined in some neighborhood of $v_0$. The value of the constant $u^{\gamma-\alpha}(H+\gamma u^{-1}u_\nu)$ is thus obtained through the chain rule:
    \begin{equation}\label{eqn:DefineA'}
        A'(v_0)=\frac{\d A}{\d V}(v_0)=\frac{\frac{\d A}{\d t}(0)}{\frac{\d V}{\d t}(0)}=\frac{\int_{\p E} u^\alpha\varphi\cdot u^{\gamma-\alpha}(H+\gamma u^{-1}u_\nu)}{\int_{\p E} u^\alpha\varphi}=u^{\gamma-\alpha}\left(H+\gamma u^{-1}u_\nu\right).\ 
    \end{equation}
    Observe that
    \begin{equation}\label{eqn:NoNonlinearity}
    \gamma u^{\alpha-1}u_\nu\varphi^2=\gamma u^{\alpha-2\gamma-1}u_\nu,\qquad
    u^\alpha\varphi\varphi_\nu=-\gamma u^{\alpha-2\gamma-1}u_\nu. 
    \end{equation}
    Applying our choice and \eqref{eqn:NoNonlinearity} to \eqref{eqn:SecondVariation}, we obtain
    \[\begin{aligned}
        \frac {\d^2 A}{\d t^2}(0) &= \int_{\p E} -\Delta_{\p E}(u^{-\gamma})-\Ric(\nu,\nu)u^{-\gamma}-|\mathrm{II}|^2u^{-\gamma}-\gamma u^{-2-\gamma}u_\nu^2 \\
        &\qquad\qquad +\gamma u^{-\gamma-1}\Big(\Delta u-\Delta_{\p E} u-Hu_\nu\Big)-\gamma u^{-1}\metric{\D_{\p E} u}{\D_{\p E} (u^{-\gamma})} \\
        &\qquad\qquad +\left(Hu^{\alpha-2\gamma}\right)u^{\gamma-\alpha}(H+\gamma u^{-1}u_\nu).
    \end{aligned}\]
    Integrating by parts and re-grouping we have
    \[\begin{aligned}
        \frac{\d^2 A}{\d t^2}(0) &= \int_{\p E} -\Ric(\nu,\nu)u^{-\gamma}+\gamma u^{-\gamma-1}\Delta u 
        + \Big[\gamma(-\gamma-1)+\gamma^2\Big]|\D_{\p E} u|^2u^{-\gamma-2} \\
        &\qquad\qquad -|\mathrm{II}|^2u^{-\gamma}-\gamma u^{-\gamma-2}u_\nu^2-\gamma Hu^{-\gamma-1}u_\nu 
        + Hu^{-\gamma}(H+\gamma u^{-1}u_\nu). 
    \end{aligned}\]
    For convenience, set $X=u^{\alpha-\gamma}A'(v_0)$, $Y=u^{-1}u_\nu$. Thus $H=X-\gamma Y$ by \eqref{eqn:DefineA'}. Using the main condition \eqref{eq:aux1}, the fact $\gamma(-\gamma-1)+\gamma^2\leq 0$, the trace inequality $|\mathrm{II}|^2\geq H^2/(n-1)$, we reduce the above inequality to
    \begin{equation}\label{eqn:EstimateSecondVariation}
        \begin{aligned}
            \frac{\d^2 A}{\d t^2}(0) &\leq \int_{\p E} -(n-1)\lambda u^{-\gamma}+u^{-\gamma}\Big[-\frac{X^2}{n-1}+\frac{2\gamma XY}{n-1}-\frac{\gamma^2 Y^2}{n-1}-\gamma Y^2 \\
            &\qquad\qquad -\gamma XY+\gamma^2 Y^2+X^2-\gamma XY\Big].
        \end{aligned}
    \end{equation}
    Next, by using the chain rule and the formulae for the derivatives of the inverse function, \begin{equation}\label{eqn:ChainRule}
        A''(v_0)= \left(\frac{\d V}{\d t}(0)\right)^{-2}\frac{\d^2 A}{\d t^2}(0) - \left(\frac{\d V}{\d t}(0)\right)^{-3}\frac{\d A}{\d t}(0)\cdot \frac{\d^2 V}{\d t^2}(0).
    \end{equation}
    Thus, joining \eqref{eqn:ChainRule}, \eqref{eqn:EstimateSecondVariation}, \eqref{eqn:VariationsVolume}, and \eqref{eqn:FirstVariationOfArea} we get, calling $Q:=\int_{\p E} u^{\alpha-\gamma}$,
    \begin{align}
        Q^2&\cdot A''(v_0) \leq \int_{\p E} -(n-1)\lambda u^{-\gamma}+u^{-\gamma}\Big[-\frac{X^2}{n-1}+\frac{2\gamma XY}{n-1}-\frac{\gamma^2 Y^2}{n-1}-\gamma Y^2 \nonumber\\
        &\qquad\qquad\qquad\qquad -\gamma XY+\gamma^2 Y^2+X^2-\gamma XY\Big] \nonumber\\
        &\qquad\qquad\quad - Q^{-1}
            \cdot \Big(\int_{\p E}u^{\alpha-\gamma}A'(v_0)\Big)
            \cdot \Big(\int_{\p E} (X+\alpha Y-2\gamma Y)u^{\alpha-2\gamma}\Big) \nonumber\\
        & = \int_{\p E} -(n-1)\lambda u^{-\gamma}+u^{-\gamma}\Big[-\frac{X^2}{n-1}+\frac{2\gamma XY}{n-1}-\frac{\gamma^2 Y^2}{n-1} \nonumber\\
        &\qquad\qquad -\gamma Y^2-\gamma XY+\gamma^2 Y^2+X^2-\gamma XY-X^2-\alpha XY+2\gamma XY\Big] \nonumber\\
        &=\int_{\p E} u^{-\gamma}\left[-\frac{X^2}{n-1}+ \left(\frac{2\gamma}{n-1}-\alpha\right)XY+\left(\frac{n-2}{n-1}\gamma^2-\gamma\right)Y^2\right] -(n-1)\lambda u^{-\gamma} \nonumber\\
        &\leq \int_{\p E} -\frac{X^2}{n-1}u^{-\gamma} -(n-1)\lambda u^{-\gamma} \label{eq:aux3}\\
        &= \int_{\p E}-\frac{A'^2(v_0)}{n-1}u^{2\alpha-3\gamma}-(n-1)\lambda u^{-\gamma}, \nonumber
    \end{align}
    where in \eqref{eq:aux3} we used $\alpha=\frac{2\gamma}{n-1}$, and $0\leq \gamma\leq\frac{n-1}{n-2}$. Notice further that $u^{-\gamma}\geq u^{2\alpha-3\gamma}$, which is implied by $\alpha\leq\gamma$ and $u\geq1$. Thus we obtain
    \begin{equation}
        Q^2A''(v_0) \leq -\Big(\frac{A'^2(v_0)}{n-1}+(n-1)\lambda\Big)\int_{\p E}u^{2\alpha-3\gamma}. \label{eqn:Final}
    \end{equation} 
    By Holder's inequality (recall $A(v_0)=\int_{\p E}u^\gamma$), we have
    \begin{equation}\label{eqn:Final2}
        A(v_0)\int_{\p E}u^{2\alpha-3\gamma}\geq\Big(\int_{\p E}u^{\alpha-\gamma}\Big)^2=Q^2.
    \end{equation}
    By joining \eqref{eqn:Final} and \eqref{eqn:Final2}, we finally obtain
    \begin{equation}\label{eqn:BarrierEquation}
        A(v_0)A''(v_0)\leq-\frac{A'(v_0)^2}{n-1}-(n-1)\lambda.
    \end{equation}
    Notice that $I(v_0)=A(v_0)$ and $I(v)\leq A(v)$ in some neighborhood of $v_0$, so $A$ is an upper barrier of $I$ which satisfies \eqref{eqn:BarrierEquation}. This is enough to prove that \eqref{eq:viscosity_aux} holds in the viscosity sense at $v_0$: indeed, for any test function $\varphi:\mathcal{I}_{v_0}\to\mathbb R$ in a neighborhood $\mathcal{I}_{v_0}\ni v_0$, we have $\varphi(v_0)=A(v_0)$ and $\varphi\leq A$ in a smaller neighborhood of $v_0$, and thus $\varphi'(v_0)=A'(v_0)$ and $\varphi''(v_0)\leq A''(v_0)$. Hence, \eqref{eqn:BarrierEquation} implies \eqref{eqn:ViscosityToShow}, as desired.
\end{proof}

\begin{lemma}\label{lem:Asymp}
    Suppose $M$ is complete, and $u\in C^\infty(M)$ is positive. Additionally, assume $x\in M$ satisfies $u(x)=\inf(u)=1$. Then, if $I$ is defined as in \eqref{eq:weighted_ip}, we have
    \[
    \limsup_{v\to0} v^{-\frac{n-1}n}I(v)\leq n\vol(\mathbb B^n)^{1/n},
    \]
    where $\mathbb B^n$ is the unit ball in $\mathbb R^n$.
\end{lemma}
\begin{proof}
    For a small $r_0$, both functions $V(r)=\int_{B(x,r)}u^\alpha$ and $A(r)=\int_{\p B(x,r)}u^\gamma$ are smooth and increasing in $(0,r_0)$. We have the asymptotics
    \[
    V(r)=\vol(\mathbb B^n)r^n+O(r^{n+1}),\qquad A(r)=n\vol(\mathbb B^n) r^{n-1}+O(r^n).
    \]
    This implies that the function $A\circ V^{-1}$ has the asymptotic
    \[A\circ V^{-1}(v)=n\vol(\mathbb B^n)^{1/n}v^{\frac{n-1}n}+o(v^{\frac{n-1}n}).\]
    The result follows from the straightforward bound $I(v)\leq A\circ V^{-1}(v)$. 
\end{proof}

\begin{lemma}\label{lem:AnotherComparison}
    Let $V\in (0,+\infty]$, and let $I:[0,V)\to \mathbb R$ be a continuous function such that $I(0)=0$, and $I(v)>0$ for every $v\in (0,V)$. Assume that for some $\lambda>0$ we have
    \[
    I''I\leq -\frac{(I')^2}{n-1}-(n-1)\lambda, \qquad \text{in the viscosity sense on $(0,V)$},
    \]
    and 
    \begin{equation}\label{eq:small_vol_limit}
        \limsup_{v\to 0^+} v^{-\frac{n-1}{n}}I(v)\leq n\vol(\mathbb B^n)^{1/n}.
    \end{equation}
    Then $V\leq \lambda^{-n/2}\vol(\mathbb S^{n})$, where $\mathbb S^n$ denotes the unit sphere in $\mathbb R^{n+1}$.
\end{lemma}
\begin{proof}
    The proof is classical. Let us write it here for completeness. Set $\psi:=I^{\frac{n}{n-1}}$. A direct computation gives that
    \begin{equation}\label{eqn:PsiZetaViscosity}
    \psi'' \leq -\lambda n\psi^{(2-n)/n} \qquad \text{in the viscosity sense on $(0,V)$.}
    \end{equation}
    For every $\zeta>0$, let us consider the model function $I_\zeta:(0,V_\zeta)\to\mathbb R$ implicitly defined by
    \begin{equation}\label{eqn:DefnIzetaNew}
        I_\zeta\Big(\zeta\int_0^r\mu(s)^{n-1}\,\d s\Big)=\zeta\mu(r)^{n-1},
    \end{equation}
    where $\mu:[0,\frac\pi{\sqrt\lambda}]\to\mathbb R$ is defined by $\mu(r):=\frac{1}{\sqrt\lambda}\sin(\sqrt\lambda r)$. Notice $V_\zeta=\frac{\zeta}{\lambda^{n/2}}\frac{\vol(\mathbb S^n)}{\vol(\mathbb S^{n-1})}$, and $I_\zeta(0)=I_\zeta(V_\zeta)=0$. A direct computation gives that the function $\psi_\zeta:=I_\zeta^{\frac{n}{n-1}}$ satisfies 
    \begin{equation}\label{eqn:PsiZeta}
    \psi_\zeta'' = -\lambda n\psi_\zeta^{(2-n)/n} \qquad \text{on $(0,V_\zeta)$.}
    \end{equation}
    Denoting by $\psi'_+(0)$ the upper right-hand Dini derivative, by \eqref{eq:small_vol_limit} we have
    \[
    \psi'_+(0)\leq n^{\frac{n}{n-1}}\vol(\mathbb B^n)^{\frac1{n-1}}.
    \]
    Moreover, direct computation gives 
    \[
    (\psi_\zeta)'_+(0)=n\zeta^{\frac1{n-1}}.
    \]
    Notice that when $\zeta=n\vol(\mathbb B^n)=\vol(\mathbb S^{n-1})$, we exactly have $V_{\zeta}=\lambda^{-n/2}\vol(\mathbb S^n)$.
    
    Suppose by contradiction that $V>\lambda^{-n/2}\vol(\mathbb S^n)$. Then there is $\zeta'>n\vol(\mathbb B^n)$ such that $V_{\zeta'}<V$. Since $\zeta'>n\vol(\mathbb B^n)$ we have 
    \begin{equation}\label{eqn:ConditionOnDerivative}
    +\infty>(\psi_{\zeta'})'_+(0) > \psi'_+(0).
    \end{equation}
    Taking into account \eqref{eqn:PsiZetaViscosity}, \eqref{eqn:PsiZeta}, and \eqref{eqn:ConditionOnDerivative} we can apply \cite[Théorème C.2.2]{BaylePhD}, after having used \cite[Proposition A.3, (1)$\Leftrightarrow$(3)]{PozzettaSurvey}. Hence we get 
    \[
    \psi(v)\leq \psi_{\zeta'}(v) \qquad \text{for all $v \leq V_{\zeta'}$}.
    \]
    The latter inequality, evaluated at $v=V_{\zeta'}$, gives $\psi(V_{\zeta'})\leq 0$, which is a contradiction with the fact that $I>0$ on $(0,V_{\zeta'})\subset (0,V)$.
\end{proof}

\begin{proof}[Proof of \cref{1mainthm} of \cref{cor:CLMSArbitraryDimension}]
    Let us assume we are in the hypotheses of \cref{cor:CLMSArbitraryDimension}, and take $u$ as in \eqref{eqn:EnhancedSpectralCondition}. We may divide $u$ by $\min(u)$, and thus assume $\min(u)=1$. Let $\pi:\tilde M\to M$ be the universal cover of $M$. Set $\tilde g=\pi^*g$ and $\tilde u:=u\circ\pi$, thus \eqref{eqn:EnhancedSpectralCondition} pulls back to give
    \begin{equation}\label{eq:spectral_cover}
        \gamma \tilde{\Delta} \tilde u \leq \tilde u \widetilde{\mathrm{Ric}}-(n-1)\lambda\tilde u.
    \end{equation}
    Let $\tilde I(v)$ be the weighted isoperimetric profile, defined as in \eqref{eq:weighted_ip}, but with $(\tilde M,\tilde g)$ and $\tilde u$ in place of $(M,g)$ and $u$. Set $\tilde V_0:=\int_{\tilde M}\tilde u^\alpha$. Since in \cref{2mainthm} of \cref{cor:CLMSArbitraryDimension} we have shown that $\tilde M$ is compact, by the classical Geometric Measure Theory, for every $v\in (0,\tilde V_0)$ there is a bounded set of finite perimeter $\tilde E\subset \tilde M$ such that $\int_{\tilde E}\tilde u^\alpha=v$ and $\int_{\p^* \tilde E}\tilde u^\gamma=\tilde I(v)$. It follows by Lemma \ref{lem:Viscosity} that we have the viscosity inequality
    \[
    \tilde I''\tilde I\leq -\frac{(\tilde I')^2}{n-1}-(n-1)\lambda \qquad \text{in }(0,\tilde V_0),
    \]
    and by Lemma \ref{lem:Asymp} that
    \[
    \limsup_{v\to 0^+} v^{-\frac{n-1}{n}}\tilde I(v)\leq n\vol(\mathbb B^n)^{1/n}.
    \]
    Also, we get that $\tilde I$ is continuous (verbatim as in \cite[Proposition 5.3]{CLMS}). Thus, a direct application of \cref{lem:AnotherComparison}, and the fact that $\min(\tilde u)=1$, gives the desired volume bound
    \begin{equation}\label{eq:vol_ineq}
        \vol(\tilde M)\leq\int_{\tilde M}\tilde u^\alpha=\tilde V_0\leq \lambda^{-n/2}\vol(\mathbb S^n).
    \end{equation}
    Finally, suppose equality holds in \eqref{eqn:VoluemControl}. Then it follows from \eqref{eq:vol_ineq} that $\tilde u$ is constant, since we have $\vol(\tilde M)=\lambda^{-n/2}\mathrm{vol}(\mathbb S^n)$, and $\min(\widetilde u)=1$. This then implies $\widetilde{\mathrm{Ric}}\geq (n-1)\lambda$, due to \eqref{eq:spectral_cover}. Hence by the rigidity of the Bishop--Gromov volume comparison theorem, $\tilde M$ is the round sphere of radius $\lambda^{-1/2}$.
\end{proof}

\begin{proof}[Proof of \cref{Corollary}]
    The compactness of $M$ and the uniform diameter upper bound are direct consequences of \cite[Theorem 1.10]{KaiIntermediate}, which utilizes an (equally) warped $\mu$-bubble in its proof. The parameter $\beta$ there corresponds to $1/\gamma$ here. Then, since $\frac4{n-1}\leq\frac{n-1}{n-2}$ for every $n\geq 3$, we get the result in \cref{Corollary} as a consequence of the main result in \cref{cor:CLMSArbitraryDimension}. The result in \cite[Theorem 1.10]{KaiIntermediate} is stated only for $3\leq n\leq 7$ for the reason of minimal surface singularities. But the result holds for $n\geq 8$ as well, by arguing similarly as in the second part of Appendix \ref{sec:Case8}.
\end{proof}

\begin{remark}[Direct isoperimetry comparison on $\tilde M$]\label{rmk:direct_isop}    The above proof of \cref{1mainthm} of \cref{cor:CLMSArbitraryDimension} is based on the fact that $\tilde M$ is compact, which was proved in \cref{2mainthm}, so that we can always find a minimizer for \eqref{eq:weighted_ip}. We remark that the existence of a minimizer for \eqref{eq:weighted_ip} on $\tilde M$ can also be proved independently of \cref{2mainthm}, as a consequence of \cite[Theorem 5.3]{NovagaPaoliniStepanovTortorelli}. The latter is proved via a concentration-compactness argument, making use of the deck transformation group. Also, the function $I(v)$ is $(1-1/n)$-H\"older continuous by arguing similarly as in \cite[Lemma 2.23]{AntonelliNardulliPozzetta}. Combining these information and Lemma \ref{lem:Viscosity}\,$\sim$\,\ref{lem:AnotherComparison}, we obtain another proof of the compactness of $\tilde M$, hence the finiteness of $\pi_1(M)$.
\end{remark}

\begin{remark}[On the supercritical case]\label{rem:Supercritical}     For the values $\gamma\in(\frac{n-1}{n-2},\infty)$, none of the results in \cref{cor:CLMSArbitraryDimension} continue to hold. A counterexample is the following. On $M=\mathbb S^1\times \mathbb S^{n-1}$ consider the spherically symmetric metric $g=\mathrm{d}r^2+\varepsilon^2f(r)^2\mathrm{d}t^2$, where $f(r)$ is a fixed positive periodic nonconstant function, and $\varepsilon>0$ will be chosen small with respect to $f$. We have
     \[
    \mathrm{Ric}(\partial_r,\partial_r)= -(n-1)\frac{f''}{f}, \qquad  \mathrm{Ric}(e,e)=-\frac{f''}{f}+(n-2)\frac{1-\varepsilon^2 (f')^2}{\varepsilon^2f^2},
    \]
    where $e$ is an arbitrary unit vector perpendicular to $\partial_r$. Hence, for sufficiently small $\varepsilon$ (depending on $f$), we have $\mathrm{Ric}=\mathrm{Ric}(\partial_r,\partial_r)$. Recall that, if $u:=u(r)$, one can directly compute
    \[
    \Delta u = u'' + \frac{n-1}{\varepsilon f}u'(\varepsilon f)'=u''+(n-1)\frac{u'f'}{f}.
    \]
    Consider the choice $u:=f(r)^{2-n}$, and set $\gamma_0=\frac{n-1}{n-2}$. A direct computation shows
    \[
    -\gamma_0\Delta u+\Ric u=-\gamma_0\Delta u+\Ric(\p_r,\p_r)u=0.
    \]
    Note that this example appears as a rigidity case of \cite[Proposition 3.5]{BourCarron}.
    
    We are now committed to show that, for all $\gamma>\gamma_0$, there exists a $\lambda>0$ such that 
    \begin{equation}\label{eqn:PositiveDirchletWanted}
    -\gamma\Delta+\mathrm{Ric}-\lambda\geq 0.
    \end{equation}
    Therefore $(M,g)$ satisfies the main condition \eqref{eqn:EnhancedSpectralCondition} (with $\frac\lambda{n-1}$ in place of $\lambda$), but it has $\pi_1(M)=\mathbb{Z}$. Let $v\in C^{\infty}(M)$, and set $\eta:=v/u$, where $u$ is as above. Then $v=u\eta$, and integrating by parts
    \begin{equation}\label{eqn:ComputationPositivityFromExistenceu}
    \begin{aligned}
        \int_M\gamma_0|\D v|^2+\Ric v^2 &= \int_M\gamma_0 u^2|\D\eta|^2+2\gamma_0 u\eta\metric{\D u}{\D\eta}+\gamma_0\eta^2|\D u|^2+\Ric u^2\eta^2 \\
        &=\int_M\gamma_0u^2|\nabla\eta|^2+u\eta^2(\Ric u-\gamma_0\Delta u) = \int_M\gamma_0u^2|\nabla\eta|^2.
    \end{aligned}
    \end{equation}
    Hence 
    \begin{equation}\label{eq:super_coercive}
    \begin{aligned}
        \int_M\gamma|\D v|^2+\Ric v^2 &= (\gamma-\gamma_0)\int_M|\D v|^2+\gamma_0\int_Mu^2|\D\eta|^2\\
        &\geq (\gamma-\gamma_0)\int_M|\nabla v|^2+\gamma_0\min(u)^2\int_M\Big|\nabla\left(\frac{v}{u}\right)\Big|^2.
    \end{aligned}
    \end{equation}
    Call $\alpha:=\gamma-\gamma_0$, $\beta:=\gamma_0\min(u)^2$. Observe that 
    \[
    c(M):=\inf\left\{\alpha\int_M|\nabla v|^2+\beta\int_M\Big|\nabla\left(\frac{v}{u}\right)\Big|^2:v\in C^\infty(M),\,\int_M v^2=1\right\}>0,
    \]
    since $u$ is nonconstant and a minimizer exists. Joining the latter with \eqref{eq:super_coercive} we finally infer that 
    \[
    -\gamma\Delta +\mathrm{Ric} - c\geq 0,
    \]
    which is the sought \eqref{eqn:PositiveDirchletWanted} with $\lambda:=c(M)>0$. \qed
\end{remark}

\begin{remark}[Smoothing the eigenfunctions]\label{rmk:smoothing_proof} 
    This remark is to prove what claimed in \cref{rmk:smoothing} in the introduction. Given $\lambda>0$ and $0\leq \gamma\leq \frac{n-1}{n-2}$, assume the positivity of Dirichlet energy
    \begin{equation}\label{eqn:PositiveSchrodinger2}
        \lambda_1(-\gamma\Delta+\Ric)\geq(n-1)\lambda,
    \end{equation}
    on the $n$-dimensional manifold $M$.
    We first aim at proving the following: \begin{equation}\label{eqn:SpectralMainThm}
        \begin{aligned}
            & \text{If $M$ is compact and \eqref{eqn:PositiveSchrodinger2} holds, then} \\
            &\hspace{108pt} \text{$\vol(\tilde M)\leq \lambda^{-n/2}\vol(\mathbb S^n)$, and $\pi_1(M)$ is finite.}
        \end{aligned}
    \end{equation}
    By regularizing the function $\mathrm{Ric}$, for each $0<\varepsilon\ll\lambda$ there exists $V_\varepsilon\in C^\infty(M)$ such that $\|V_\varepsilon-\mathrm{Ric}\|_{C^0}\leq \varepsilon$. Note that $\lambda_1(-\gamma\Delta+V_\epsilon)\geq(n-1)\lambda-\epsilon$. Let $u_\epsilon\in C^\infty(M)$ be the first eigenfunction of $-\gamma\Delta+V_\epsilon$. Notice that $u_\varepsilon$ is positive, and we have
    \begin{equation}\label{eq:approx_eigen}
        -\gamma\Delta u_\epsilon+\Ric u_\epsilon\geq\Big[(n-1)\lambda-2\epsilon\Big]u_\epsilon.
    \end{equation}
    Thus, applying \cref{cor:CLMSArbitraryDimension} with \eqref{eq:approx_eigen} and taking $\varepsilon\to 0$, we have proved the desired \eqref{eqn:SpectralMainThm}.
    \smallskip
    
    Let us now deal with the rigidity part. For simplicity, after rescaling, assume $\lambda=1$. Hence, assume we have $\lambda_1(-\gamma\Delta u+\Ric)\geq n-1$ and $\vol(M)=\vol(\mathbb S^n)$. Let us show that $M$ is isometric to the standard sphere of radius $1$.

    Let $u\in W^{1,2}(M)$ be the first eigenfunction of $-\gamma\Delta +\mathrm{Ric}$, so that $-\gamma\Delta u+\Ric u=\lambda u$ for some $\lambda\geq(n-1)$. Normalize $u$ so that $\int_M u^2=1$. Set $\lambda_\epsilon:=\lambda_1(-\gamma\Delta+V_\epsilon)$, and let $u_\epsilon\in C^\infty(M)$ be the corresponding first eigenfunction, normalized so that $\int_M u_\epsilon^2=1$. Note that $|\lambda-\lambda_\epsilon|\leq\epsilon$, and $u,u_\epsilon$ are all positive.

    Let us show that $u_\epsilon\to u$ in $C^{1,\beta}$ for every $0<\beta<1$. Note that $\int_M |\D u_\epsilon|^2$ are uniformly bounded, so there is a subsequence converging to some $\tilde u\in W^{1,2}$ strongly in $L^2$ and weakly in $W^{1,2}$. Using lower semi-continuity, and the fact that $u_\epsilon$ are normalized eigenfunctions, we have:
    \[
    \int_M\gamma|\D\tilde  u|^2+\Ric\tilde  u^2\leq\liminf_{\epsilon\to0}\int_M\gamma|\D u_\epsilon|^2+\lim_{\epsilon\to0}\int_M V_\epsilon u_\epsilon^2\leq\lim_{\epsilon\to0}\lambda_\epsilon=\lambda.
    \]
    So it must happen that $\tilde u=u$. Thus $u_\epsilon\to u$ in $L^2$. Now, applying standard elliptic regularity, this can be bootstrapped to imply $u_\epsilon\to u$ in $W^{2,p}$ for every $p>1$, hence in $C^{1,\beta}$ for every $0<\beta<1$.
    
    By \cref{cor:CLMSArbitraryDimension} (specifically, \eqref{eq:vol_ineq} in its proof) and the rigidity assumption, we have
    \[\begin{aligned}
        \vol(M) &\leq \min(u_\epsilon)^{-\alpha}\int_M u_\epsilon^\alpha 
        \leq \vol(\mathbb S^{n})\big(1+o_\epsilon(1)\big)=\vol(M)\big(1+o_\epsilon(1)\big).
    \end{aligned}\]
    Passing this to the limit, we get
    \[\vol(M)\leq\min(u)^{-\alpha}\int_M u^\alpha\leq\vol(M).\]
    Hence $u$ is constant, and we are reduced to pointwise Ricci bound, from which the rigidity follows from the classical Bishop-Gromov volume comparison.

    Finally, the Bonnet-Myers result \eqref{eqn:BonnetMyers} follows from \eqref{eq:approx_eigen} and the $C^{1,\beta}$-convergence obtained above. 
    \qed
\end{remark}

\section{Isoperimetry under \texorpdfstring{$\mathrm{Ric}\geq 0$}{Ric≥0} and spectral \texorpdfstring{$\mathrm{biRic}\geq n-2$}{biRic≥n-2}}\label{sec:isop}

We aim at showing that when $3\leq n\leq 5$, a spectral biRicci condition descends to a spectral Ricci condition on properly chosen warped $\mu$-bubbles. We will keep track of the sharp constants in the process, and we will use the forthcoming computations to show \cref{existencegoodmububble}. We stress that the following computations are inspired by \cite[Theorem 4.3]{CLMS}. Similar computations have appeared in \cite[Theorem 4.1]{Mazet} in the proof of the stable Bernstein problem in $\mathbb R^6$, while we were revising the final draft of this paper.

Let $(M^n,g)$ be a complete smooth Riemannian manifold. Let $3\leq n\leq 5$, and $0\leq \gamma<6-n$. It will also be clear from the computations that when $n=4$ we can choose $0\leq \gamma\leq 2$. Let $u$ be a positive smooth function on $M$ such that, for some $\lambda\in\mathbb R$, 
\begin{equation}\label{eqn:biRicbound}
-\gamma\Delta u + u\mathrm{biRic}\geq \lambda u.
\end{equation}
For an arbitrary fixed domain $\Omega_0$ with $\Omega_-\subset\subset\Omega_0\subset\subset\Omega_+\subset M$, consider the functional
\begin{equation}\label{eqn:MuBubbleFunctional2}
    E(\Omega):=\int_{\p^*\Omega}u^\gamma-\int(\chi_\Omega-\chi_{\Omega_0})hu^\gamma,
\end{equation}
defined on sets of finite perimeter, where $\p^*\Omega$ is the reduced boundary of $\Omega$. Take $\Omega$ to be a minimizer, and notice $\partial\Omega$ is smooth because $3\leq n\leq 5$.

Let $\nu$ denote the outer unit normal at $\p\Omega$, and let $\varphi\in C^\infty(M)$. For an arbitrary smooth variation $\{\Omega_t\}_{t\in(-\varepsilon,\varepsilon)}$ with $\Omega_0=\Omega$ and variational field $\varphi\nu$ at $t=0$, we compute the first variation
\begin{equation}
    0=\frac{\mathrm{d}}{\mathrm{d}t}E(\Omega_t)\Big|_{t=0}=\int_{\p\Omega}\big(H+\gamma u^{-1}u_\nu-h\big)u^\gamma\varphi.
\end{equation}
Since $\varphi$ is arbitrary we have $H=h-\gamma u^{-1}u_\nu$. Then, computing the second variation,
\[\begin{aligned}
    0 &\leq \frac {\mathrm{d}^2}{\mathrm{d}t^2}E(\Omega_t)\Big|_{t=0} \\
    &= \int_{\p\Omega}\Big[-\Delta_{\p\Omega}\varphi-|\text{II}|^2\varphi-\Ric(\nu,\nu)\varphi-\gamma u^{-2}u_\nu^2\varphi \\
    &\hspace{72pt} +\gamma u^{-1}\varphi\big(\Delta u-\Delta_{\p\Omega}u-Hu_\nu\big)-\gamma u^{-1}\metric{\D_{\p\Omega}u}{\D_{\p\Omega}\varphi} -h_\nu \varphi\Big]u^\gamma\varphi.
\end{aligned}\]
Set $\varphi:=\psi u^{-\gamma/2}$, so $u^\gamma\varphi^2=\psi^2$. Then we get 
\begin{equation}\label{eqn:ControlToDo2}
    \begin{aligned}
        0 &\leq \int_{\p\Omega}\Big[
            -\Delta_{\p\Omega}(\psi u^{-\gamma/2})\psi u^{\gamma/2}
            -\gamma u^{-1}\psi^2\Delta_{\partial\Omega}u
            -\gamma u^{\gamma/2-1}\psi  \metric{\D_{\p\Omega}u}{\D_{\p\Omega}(\psi u^{-\gamma/2})}\Big] \\
        &\qquad\qquad +\psi^2\Big[-|\text{II}|^2-\Ric(\nu,\nu)-\gamma u^{-2}u_\nu^2 
        +\gamma u^{-1}\big(\Delta u-Hu_\nu\big)
        -h_\nu\Big] \\
        &=: P+Q.
    \end{aligned}
\end{equation}
Recall that, see \cite[Page 13]{CLMS},
\begin{equation}\label{eqn:RicBiRic}
    |\text{II}|^2+\mathrm{Ric}(\nu,\nu)\geq \mathrm{biRic}-\mathrm{Ric}_{\p\Omega}+\frac{6-n}{4}H^2.
\end{equation}
Let us call $Y=u^{-1}u_\nu$, thus $H=h-\gamma Y$. By \eqref{eqn:RicBiRic} and \eqref{eqn:biRicbound} we have:
\[\begin{aligned}
    -|\text{II}|^2-\mathrm{Ric}(\nu,\nu)+\gamma u^{-1}\Delta u
    &\leq -\mathrm{biRic}+\mathrm{Ric}_{\p\Omega}-\frac{6-n}{4}H^2+\gamma u^{-1}\Delta u \\
    &\leq -\lambda +\mathrm{Ric}_{\p\Omega}-\frac{6-n}{4}H^2.
\end{aligned}\]
Thus
\begin{equation}\label{eqn:ControlOnQ}
    \begin{aligned}    
        Q &\leq\int_{\partial\Omega} \psi^2\Big[-\lambda+\mathrm{Ric}_{\p\Omega}-\frac{6-n}{4}(h-\gamma Y)^2-\gamma Y^2-\gamma(h-\gamma Y)Y+|\nabla h|\Big] = \\
        &\leq\int_{\partial\Omega} \psi^2\Big[-\lambda+\mathrm{Ric}_{\p\Omega}+|\nabla h|-\frac{6-n}{4}h^2+\frac{4-n}{2}\gamma hY+\left(\frac{n-2}{4}\gamma^2-\gamma\right)Y^2\Big].
    \end{aligned}
\end{equation}
Notice that $6-n\leq\frac4{n-2}$ for all $n>2$, with equality if and only if $n=4$. Hence, since $\gamma<6-n$, we also have $\gamma<\frac4{n-2}$. It can computed directly that
\begin{equation}
    \begin{aligned}\label{eqn:AlgebraicInequality}
    &\frac{6-n}{4}h^2-\frac{4-n}{2}\gamma h Y + \left(\gamma-\frac{n-2}{4}\gamma^2\right)Y^2\\
    &\quad =\left(\frac{6-n}{4}-\frac{(4-n)^2\gamma}{4\left(4-(n-2)\gamma\right)}\right)h^2+\left(\gamma-\frac{n-2}{4}\gamma^2\right)\left(\frac{(4-n)}{4-(n-2)\gamma}h-Y\right)^2.
    \end{aligned}
\end{equation}
Denote 
\begin{equation}\label{eqn:cngamma}
c(n,\gamma):=\frac{6-n}{4}-\frac{(4-n)^2\gamma}{4(4-(n-2)\gamma)},
\end{equation}
and notice that $\gamma<6-n \Leftrightarrow c(n,\gamma)>0$, and $c(n,\gamma)=0$ if and only if $\gamma=6-n$. Moreover, since $0\leq \gamma<\frac4{n-2}$ as noticed above, we have $\gamma-\frac{n-2}{4}\gamma^2 \geq 0$. Joining \eqref{eqn:ControlOnQ}, and \eqref{eqn:AlgebraicInequality} we thus conclude
\begin{equation}\label{eqn:IneqOnQ2}
Q\leq \int_{\partial\Omega}\psi^2\left[-\lambda+\mathrm{Ric}_{\p\Omega}+|\nabla h|-c(n,\gamma)h^2\right].
\end{equation}
Notice that in case $n=4$, and $\gamma=2$, \eqref{eqn:IneqOnQ2} still holds with $c(4,2):=1/2$ (i.e., when $n=4$ we can work in the full interval $0\leq \gamma\leq 2$ as claimed above).

Then we deal with the gradient terms $P$.
Notice that we have the following integration by parts formula for every smooth function $\eta\in C^\infty(\partial\Omega)$:
\begin{equation}\label{eqn:IntegrationByPartTricky2}
    \begin{aligned}
        &\quad \int_{\partial\Omega}\metric{\D_{\partial\Omega}(u^{-\gamma}\eta)}{\D_{\partial\Omega}\eta}-\gamma u^{-\gamma-1}\eta^2\Delta_{\partial\Omega} u-\gamma u^{-1}\eta\metric{\D_{\partial\Omega} u}{\D_{\partial\Omega}(u^{-\gamma}\eta)} \\
        &= \int_{\partial\Omega} u^{-\gamma}|\D_{\partial\Omega}\eta|^2-\gamma u^{-\gamma-2}\eta^2|\D_{\partial\Omega} u|^2.
    \end{aligned}
    \end{equation}
Thus applying \eqref{eqn:IntegrationByPartTricky2} with $\eta=\psi u^{\gamma/2}$ we get, using Young's inequality, 
\[\begin{aligned}
    P&=\int_{\partial\Omega} u^{-\gamma}|\nabla_{\partial\Omega}(\psi u^{\gamma/2})|^2-\gamma u^{-2}\psi^2|\nabla_{\partial\Omega} u|^2 \\
    &=\int_{\partial\Omega} |\nabla_{\partial\Omega} \psi|^2+\gamma u^{-1}\psi\metric{\nabla_{\partial\Omega} u}{\nabla_{\partial\Omega} \psi}+\left(\frac{\gamma^2}{4}-\gamma\right)u^{-2}\psi^2|\nabla_{\partial\Omega} u|^2  \\
    &\leq \int \left(1+\frac{\gamma^2}{4c}\right)|\nabla_{\partial\Omega}\psi|^2+\left(\frac{\gamma^2}{4}-\gamma+c\right)u^{-2}\psi^2|\nabla_{\partial\Omega} u|^2.
\end{aligned}\]
Choose $c=\gamma-\gamma^2/4$, which is positive since $0\leq \gamma\leq 4$ in the range we are working in. Hence we get
\begin{equation}\label{eqn:IneqOnS2}
S \leq \int \frac{4}{4-\gamma}|\nabla_{\partial\Omega}\psi|^2.
\end{equation}
Putting together \eqref{eqn:IneqOnQ2}, \eqref{eqn:IneqOnS2}, and \eqref{eqn:ControlToDo2} we get
\begin{equation}\label{eqn:44-gamma}
    \int_{\p\Omega}\frac4{4-\gamma}\big|\nabla_{\partial\Omega}\psi\big|^2+\Ric_{\p\Omega}\psi^2\geq\int_{\p\Omega}\Big[\lambda-|\nabla h|+c(n,\gamma)h^2\Big]\psi^2,
\end{equation}
for every $\psi\in C^\infty(\partial\Omega)$.
\smallskip

The previous computations can be used to infer the following key \cref{existencegoodmububble}. In dimension $n=4$, or when $\gamma=0$, similar results have appeared in \cite[Theorem 4.3]{CLMS}, and \cite[Lemma 3.1]{KaiIntermediate}, respectively.

\begin{lemma}\label{existencegoodmububble}
    Let $(M^n,g)$ be a smooth complete noncompact Riemannian manifold with $3\leq n\leq5$, let $0\leq \gamma<6-n$ (or $0\leq \gamma\leq 2$ when $n=4$), and $S$ be a compact, and connected set such that 
    \begin{equation}\label{eqn:AssumptionLemmaMuBubbleTwisted}
    -\gamma\Delta+\mathrm{biRic}\geq n-2
    \end{equation}
    on $M\setminus S$. Then there is a constant $C:=C(n,\gamma)$ with the following property: for any $0<\varepsilon<n-2$ there exists a connected, smooth, bounded region $\Omega\supset\supset S$, such that $d(\p\Omega,S)\leq C\varepsilon^{-1/2}$, each connected component $\Sigma$ of $\p\Omega$ satisfies
    \begin{equation}\label{eq:mububble_spec}
        -\frac{4}{4-\gamma}\Delta_\Sigma+\mathrm{Ric}_\Sigma-(n-2-\varepsilon)\geq 0,
    \end{equation}
    and $M\setminus\Omega$ does not have bounded connected components.
\end{lemma}
\begin{proof}
    Fix $0<\varepsilon<n-2$. From the assumption \eqref{eqn:AssumptionLemmaMuBubbleTwisted}, and using \cite[Theorem 1]{FischerColbrieSchoenCPAM}, we can find a positive $u\in C^\infty(M\setminus S)$ such that
    \[-\gamma\Delta u + u\biRic \geq (n-2-\varepsilon/2)u.\]
    By mollifying, we find $q\in C^\infty(M)$ such that $|\D q|\leq2$ everywhere, $q=0$ on $S$, and $q\geq0$, $|q-d(\cdot,S)|\leq1$ on $M\setminus S$. Set the constant $C:=c(n,\gamma)$ as in \eqref{eqn:cngamma}, and consider
    \[h=\sqrt{\varepsilon/(2C)}\cot\Big(\frac12\sqrt{(C\varepsilon)/2}\,q\Big).\]
    By direct calculation, it follows that
    \begin{equation}\label{eqn:InequalityDh}
    |\D h|\leq Ch^2+\varepsilon/2.
    \end{equation}
    Set $\Omega_-:=S$ and $\Omega_+:=\big\{q<2\pi/\sqrt{(C\varepsilon)/2}\big\}$. Let us consider a warped stable $\mu$-bubble $\widetilde\Omega$ associated to the energy
    \[
    E(\Omega):=\int_{\partial^*\Omega}u^\gamma-\int(\chi_\Omega-\chi_{\Omega_0})h u^\gamma,
    \]
    where $\Omega_0$ is a fixed set with $\Omega_-\subset\subset\Omega_0\subset\subset\Omega_+$. By construction we have $\Omega_-\subset\subset\widetilde\Omega\subset\subset\Omega_+$, in particular, $d(\p\widetilde\Omega,S)\leq C\varepsilon^{-1/2}$, with a possibly different $C$ only depending on $n,\gamma$.
    
    Now let us exploit the computations before the present Lemma. We have \eqref{eqn:biRicbound} with $\lambda:=n-2-\varepsilon/2$ on $M\setminus S$. Notice that $\widetilde\Omega\subset\subset M\setminus S$, so the argument before the previous Lemma, and in particular \eqref{eqn:44-gamma} together with \eqref{eqn:InequalityDh}, will let us conclude that $\p\widetilde\Omega$ (hence each of its components) satisfies the spectral condition \eqref{eq:mububble_spec}.

    Finally, we let $\Omega'$ be the connected component of $\widetilde\Omega$ containing $S$, and let $\Omega$ be the union of $\Omega'$ with all the bounded connected components of $M\setminus\Omega'$. Note that $\p\Omega$ is a sub-collection of the connected components of $\p\tilde\Omega$, hence all the conditions of the lemma are fulfilled.
\end{proof}

The following lemma comes from adapting \cite[Theorem 1.1]{ChodoshLiStryker} and \cite[Lemma 2.2]{ChodoshLiStryker} using \cref{existencegoodmububble}. In a manifold $M$ with one end, we let $E_r$ be the unique unbounded component of $M\setminus\bar{B_r(o)}$. Thus $M\setminus E_r$ is connected and bounded, and $E_r\subset E_s$ when $r>s$.

\begin{lemma}\label{lem:CentralLemmaForIsop}
    Let $(M^n,g)$ be a smooth complete noncompact Riemannian manifold with $3\leq n\leq 5$, and let $0\leq \gamma < 6-n$ (or $0\leq \gamma\leq 2$ when $n=4$). Fix $o\in M$. Assume that $\mathrm{Ric}\geq 0$, $M$ has one end, and $-\gamma\Delta+\mathrm{biRic}\geq n-2$ outside a compact set $K$. Then there are $r_0:=r_0(o,M,K)>0$, $C_0:=C_0(o,M,K)>0$, and universal constants $C,L>0$ such that the following holds.
    \begin{enumerate}
        \item For every $r> r_0$ and $0<\varepsilon<n-2$, there is a connected unbounded set $G$ such that $M\setminus G$ is connected and bounded, with $E_{r+C\varepsilon^{-1}}\subset G\subset E_r$, and such that its boundary $\Sigma$ is smooth, connected, and satisfies
        \begin{equation}\label{eqn:SpectralOnMuBubbles}
            -\frac{4}{4-\gamma}\Delta_\Sigma+\Ric_\Sigma -(n-2-\varepsilon)\geq 0.
        \end{equation}

        \item For every $r>r_0$ we have
        \begin{equation}\label{eqn:VolOfSections}
            \vol\big(E_r\setminus E_{r+L}\big)\leq C.
        \end{equation}
        \item For every $r>0$ we have 
        \begin{equation}\label{eqn:VolumeBoundC0}
        \vol(B_r(o))\leq C_0 r,
        \end{equation}
        i.e., $M$ has linear volume growth.
    \end{enumerate}
\end{lemma}
\begin{proof}
    We prove the three items independently.
    \begin{enumerate}
        \item By \cite[Proposition 1.3]{AbreschGromoll} we have that $\mathrm{Ric}\geq 0$ implies $b_1(M)<+\infty$. Then there exists a sufficiently large $r_0$ such that: for all $r>r_0$, the inclusion map $B_r(o)\hookrightarrow M$ sends $H_1(B_r(o);\RR)$ surjectively to $H_1(M;\RR)$. 
        Apply Lemma \ref{existencegoodmububble} to find a connected and bounded region $\Omega\supset\supset B_r(o)$. Denote $\Sigma=\p\Omega$ and $G=M\setminus\Omega$. Note that $\Omega,G$ are connected by \cref{existencegoodmububble}, and since $M$ has one end. Then we have (making implicit the coefficient $\RR$) the following Mayer--Vietoris exact sequence
        \[\hspace{32pt}\cdots\to H_1(\Omega)\oplus H_1(G)\to H_1(M)\to H_0(\Sigma)\to H_0(\Omega)\oplus H_0(G)\to H_0(M)\to0.\]
        Note that the map $H_1(\Omega)\to H_1(M)$ is surjective, since it is factored through $H_1(B_r(o))\to H_1(\Omega)\to H_1(M)$. Then the long exact sequence implies $H_0(\Sigma)=\RR$, hence $\Sigma$ is connected. Further observe that $G$ is connected and noncompact, and we have $\bar G\subset M\setminus B_r(o)$ and $d(\p G,B_r(o))\leq C\varepsilon^{-1}$ by \cref{existencegoodmububble}. These directly imply $E_{r+C\varepsilon^{-1}}\subset G\subset E_r$. Finally, \eqref{eqn:SpectralOnMuBubbles} comes from \eqref{eq:mububble_spec} in \cref{existencegoodmububble}.
        
        \item This follows by a direct adaption of \cite[Lemma 2.2]{ChodoshLiStryker}. Take $\varepsilon:=1/2$, and $L:=C\varepsilon^{-1}$. Indeed, by Item 1, and by possibly taking a larger $r_0$, we can first find two connected unbounded domains $G_1, G_2$ with connected boundaries, such that $E_r\subset G_1\subset E_{r-L}$ and $E_{r+2L}\subset G_2\subset E_{r+L}$. It remains to bound the volume of $G_1\setminus G_2$, and in view of Bishop--Gromov volume comparison, it is sufficient to bound the diameter of $G_1\setminus G_2$. 
        
        The latter diameter bound can be obtained following verbatim the argument in \cite[Lemma 2.2]{ChodoshLiStryker}, which involves the almost splitting theorem of Cheeger-Colding \cite{ChCo0,ChCo1}, and a uniform diameter bound of $\p G_1$, $\p G_2$. When $n\geq4$, this uniform diameter bound comes from \eqref{eqn:SpectralOnMuBubbles} and the last part of Corollary \ref{Corollary}. The Corollary is applied with the $\gamma$ there being equal to $\frac4{4-\gamma}$ here. This requires $\frac4{4-\gamma}<\frac4{n-2}$, which is satisfied by our assumption $\gamma<6-n$. Notice that when $n=4$, it suffices to have $\leq$ instead of $<$ in the previous inequalities. For the case $n=3$, we may alternatively apply \cite[Theorem 1.4]{KaiSurfaces} to get the diameter bounds of the (surfaces) $\p G_1$, $\p G_2$. When $n=3$, $\gamma<3$, and then $\frac{4}{4-\gamma}<4$, which is what we need to apply \cite[Theorem 1.4]{KaiSurfaces}.
        
        \item If $r>r_0$, this follows by decomposing
        \[\hspace{32pt}B_{r_0+kL}(o)\subset(M\setminus E_{r_0})\cup(E_{r_0}\setminus E_{r_0+L})\cup\cdots\cup(E_{r_0+(k-1)L}\setminus E_{r_0+kL})\]
        and using the result of Item 2 above. If $r<r_0$, the result is a direct consequence of Bishop--Gromov volume comparison. \qedhere
    \end{enumerate}
\end{proof}

\begin{proof}[Proof of \cref{thm:isop}]
    We prove the two items separately.
    \begin{enumerate}
        \item This is a direct consequence of Item 3 of \cref{lem:CentralLemmaForIsop}.
        \item By \cite[Theorem 3.8(1)]{ConcavitySharpAPPS2} the isoperimetric profile $I_M$ is nondecreasing, since $\mathrm{Ric}\geq 0$. Hence, in order to prove \eqref{eqn:BoundOnProfile}, it is enough to prove the following
        \begin{equation}\label{eqn:SoughtInequality}
            \text{For every $\delta,V_0>0$ there is $V>V_0$ such that $I_M(V)\leq \vol(\mathbb S^{n-1})+\delta$.}
        \end{equation}
        Now fix $\delta,V_0$. Let $o\in M$, let $K\subset M$ be a compact set such that $-\gamma\Delta+\mathrm{biRic}\geq n-2$ on $M\setminus K$, and take $r_0:=r_0(o,M,K)$ as in the statement of \cref{lem:CentralLemmaForIsop}. Fix $r>r_0$ such that $\vol(B_r(o))>V_0$, and fix $\lambda<1$ such that $\lambda^{-(n-1)/2}\vol(\mathbb S^{n-1})<\vol(\mathbb S^{n-1})+\delta$. 
        
        Applying Item 1 of \cref{lem:CentralLemmaForIsop} with the choice $\varepsilon=(n-2)(1-\lambda)$, we find a set $G$ with $E_{r+C\varepsilon^{-1}}\subset G\subset E_r$, such that $M\setminus G$ is bounded, and $\p G$ is connected and satisfies $-\frac4{4-\gamma}\Delta_{\p G}+\Ric_{\p G}-(n-2)\lambda\geq 0$. When $n=4,5$, it follows from \cref{cor:CLMSArbitraryDimension} (cf. \eqref{eqn:SpectralMainThm}) that $\Omega:=M\setminus G$ satisfies  
        \begin{equation}\label{eqn:ConclsuonbOund}
        V_0<\vol(\Omega)<+\infty, \quad \text{and} \quad \text{$\vol(\partial \Omega)<\vol(\mathbb S^{n-1})+\delta$}.
        \end{equation}
        When $n=3$ (i.e., $\p G$ is 2-dimensional), we may test the spectral condition $-\frac{4}{4-\gamma}\Delta_{\p G}+\Ric_{\p G}-\lambda\geq 0$ with the constant function $\psi\equiv1$ and use Gauss--Bonnet theorem to get
        \[\vol(\p\Omega)\leq\lambda^{-1}\int_{\p\Omega}\Ric_{\p\Omega}\leq 4\pi\lambda^{-1}=\lambda^{-1}\vol(\mathbb S^2)<\vol(\mathbb S^2)+\delta.\]
        This proves \eqref{eqn:SoughtInequality} in either case, and thus concludes the proof of \eqref{eqn:BoundOnProfile}.
        \smallskip
        
        Now let us deal with the rigidity part. If $\inf_{x\in M}\vol(B_1(x))=0$, then by \cite[Proposition 2.18]{ABFP21} we have $I_M\equiv0$. So, if equality holds in \eqref{eqn:BoundOnProfile}, $M$ is noncollapsed, i.e., $\inf_{x\in M}\vol(B_1(x))>0$. Then we can apply the results of \cite{XingyuZhu} in what follows.
        
        Now, since $I$ is nondecreasing, if equality holds for some $v>0$ in \eqref{eqn:BoundOnProfile}, then it holds for every $V\geq v$. Thus, from \cite[Theorem 4.2(3)]{XingyuZhu} one gets that the quantity $\mathfrak{D}$ in there is equal to $\vol(\mathbb S^{n-1})$. Hence from the rigidity of \cite[Theorem 4.2(2)]{XingyuZhu} we get that there is a bounded open set $\Omega\subset M$ such that $M\setminus\Omega$ is isometric to $\partial\Omega \times [0,+\infty)$. Moreover, $\vol(\partial\Omega)=\vol(\mathbb S^{n-1})$ because isoperimetric regions with large volume exist and have boundary isometric to $\partial\Omega$. From the assumption $-\gamma\Delta+\mathrm{biRic}\geq n-2$ outside a compact set we get that $-\gamma\Delta_{\partial\Omega}+\mathrm{Ric}_{\partial\Omega}\geq n-2$. Thus, from the rigidity in \cref{cor:CLMSArbitraryDimension}, see in particular \cref{rmk:smoothing_proof}, $\partial\Omega$ must be isometric to $\mathbb S^{n-1}$. The proof of Item 2 is thus concluded. \qedhere
    \end{enumerate}
\end{proof}
\begin{remark}\label{rem:rigidity}
    We include the sketch of another argument that leads to the rigidity statement in \cref{item2isop} of \cref{thm:isop}. Suppose $I_M(v_0)=\vol(\mathbb S^{n-1})$ for some $v_0>0$. Then, since $I_M$ is nondecreasing, we have $I_M(v)=\vol(\mathbb S^{n-1})$ for all $v\geq v_0$. We may increase $v_0$, and assume (using \cite[Theorem 4.2(1)]{XingyuZhu}) that a smooth isoperimetric set $\Omega$ exists with $\vol(\Omega)=v_0$, and $|\p\Omega|=I_M(v_0)$. Consider the unit speed variation $\{\Omega_t\}_{-\varepsilon<t<\varepsilon}$ with $\Omega_0=\Omega$. Since $I'_M(v_0)=0$, $\p\Omega$ is a minimal surface. Let $H_t$ be the mean curvature of $\p\Omega_t$. For $t\geq0$ we have
    \begin{equation}\label{eq:rig_aux1}
        \frac {\d}{\d t}H_t=-|\mathrm{II}_t|^2-\Ric(\nu_t,\nu_t)\leq0,
    \end{equation}
    so $H_t\leq0$ for all $t\geq0$. Hence
    \begin{equation}\label{eq:rig_aux2}
        \frac {\d}{\d t}|\p\Omega_t|=\int_{\p\Omega_t}H_t\leq0.
    \end{equation}
    On the other hand, we must have $|\p\Omega_t|\geq I_M(\vol(\Omega_t))=I_M(v_0)=\vol(\mathbb S^{n-1})=|\p\Omega_0|$. Therefore, \eqref{eq:rig_aux1} \eqref{eq:rig_aux2} must achieve equality for all $t>0$, and in particular, all the $\Omega_t$ are isoperimetric sets as well. Thus the unit speed variation can be extended indefinitely, and we obtain a totally geodesic foliation with vanishing normal Ricci curvature. This implies a splitting of the exterior region.
\end{remark}

\begin{remark}\label{rem:volumegrowth}
    In the papers \cite{WeiXuZhang, HuangLiu}, which appeared when we were completing a first draft of this paper, the authors prove a sharp volume growth result at infinity when $n=3$, and $\gamma=0$ in \cref{thm:isop}. 
    
    Joining our result in \cref{thm:isop}, and using a different limit space argument, we can prove the following. Let $(M^n,g)$ be a one-ended complete Riemannian manifold with $n\in \{3,4,5\}$, $0\leq \gamma< 6-n$ (or $0\leq \gamma\leq 2$ when $n=4$), $\mathrm{Ric}\geq 0$, and $\lambda_1(-\gamma\Delta+\mathrm{biRic})\geq n-2$ outside a compact set. Assume further \begin{equation}\label{eqn:noncollapse}
        \inf_{p\in M}\vol(B_1(p))>0.
    \end{equation}
    Then, for any $o\in M$ it holds
\begin{equation}\label{eqn:VolumeGrowthResult}
\limsup_{r\to+\infty}\frac{\mathrm{\vol(B_r(o))}}{r}\leq \vol(\mathbb S^{n-1}).
    \end{equation}
    
    Note that in dimension $n=3$, and when $\gamma=0$, \eqref{eqn:noncollapse} follows from the other assumptions, see the recent \cite[Proposition 1.10]{HuangLiu}. In general, it is likely that the assumption \eqref{eqn:noncollapse} is not needed: e.g., by joining \cref{cor:CLMSArbitraryDimension}, the Jacobian comparison argument in \cite[Lemma 5.3]{WeiXuZhang}, and arguing as in \cite{WeiXuZhang}; or by proving a more general version of \cite[Proposition 1.10]{HuangLiu}.

    Let us present how to deduce \eqref{eqn:VolumeGrowthResult} from the results in \cite{XingyuZhu} under the non-collapsing assumption \eqref{eqn:noncollapse}. With \eqref{eqn:noncollapse} available, we are free to use the results of \cite{XingyuZhu}. Notice that $M$ has linear volume growth by \cref{item1isop} in \cref{thm:isop}. In particular, by combining  \cite[Corollary 4.3]{XingyuZhu}, and \cref{item2isop} in \cref{thm:isop} we conclude
    \[
        \limsup_{r\to+\infty}\frac{\vol(B_r(o))}{r}=\lim_{v\to\infty}I_M(v)\leq \vol(\mathbb S^{n-1}).
    \]
\end{remark}

\appendix

\section{The case \texorpdfstring{$n\geq 8$}{n≥8}}\label{sec:Case8}

When trying to prove Theorem \ref{cor:CLMSArbitraryDimension} in dimensions $n\geq8$, one encounters the issue that the minimizer in Lemma \ref{lem:SecondPointDiameter} or Lemma \ref{lem:Viscosity} may contain singularities. In this section, we present a way to modify the second variation argument so that it is applicable to possibly singular minimizers. Similar arguments have also appeared in \cite{Bayle04, BraySingular, JonathanZhu}.

\begin{proof}[Proof of Lemma \ref{lem:Viscosity} ($n\geq8$)]\label{proofnonsmooth} {\ }
    
    Let $E\subset M$ be a bounded minimizer such that $V(E)=v_0$ and $A(E)=I(v_0)$. Let $K$ be a compact set with $E\subset K$.
    By the classical Geometric Measure Theory (see \cite[Section 3.10]{MorganRegular}), the regular part of $\p E$ (denoted by $\preg E$) is a smooth hypersurface, while the singular part of $\p E$ (denoted by $\psing E$) has Hausdorff dimension at most $n-8$.

    For each $\delta<1/4$, we can find a finite collection of balls $B(x_i,r_i)$ with $x_i\in \psing E$ and $r_i<\delta$, such that $\sum r_i^{n-7}\leq1$. For each $i$, we find a smooth function $\eta_i$ such that
    \[\eta_i|_{B(x_i,2r_i)}=0,\quad \eta_i|_{M\setminus B(x_i,3r_i)}=1,\quad |\D_M\eta_i|\leq 2r_i^{-1}.\]
    
    We claim that for each $x\in K$ and $r<1$ we have
    \begin{equation}\label{eq:area_in_ball}
        \int_{\p^*E\cap B(x,r)}u^\gamma\leq Cr^{n-1},
    \end{equation}
    where $C$ depends only on $K$ and $u$. The constant $C$ might change from line to line from now on. To see this, for each $x\in M$ and $r>0$ there is a radius $s\in[0,r]$ such that $\int_{B(x,r)\setminus B(x,s)}u^\alpha=\int_{B(x,r)\cap E}u^\alpha$. This implies that the set
    \[
    E'=(E\cup B(x,r))\setminus B(x,s)
    \]
    has $V(E')=V(E)$. On the other hand, we have
    \[A(E')\leq \int_{\p^*E\setminus\bar{B(x,r)}}u^\gamma+\int_{\p^*B(x,r)}u^\gamma+\int_{\p^*B(x,s)}u^\gamma\leq\int_{\p^*E\setminus\bar{B(x,r)}}u^\gamma+Cr^{n-1}\]
    and
    \[A(E')\geq A(E)\geq\int_{\p^*E\setminus\bar{B(x,r)}}u^\gamma+\int_{\p^*E\cap B(x,r)}u^\gamma.\]
    This proves \eqref{eq:area_in_ball}. By regularizing $\overline\eta:=\min_i\{\eta_i\}$, we can find a function $\eta\in C^\infty(M)$ such that
    \[\eta=0\ \text{on}\ \bigcup B(x_i,r_i),\qquad
        \eta=1\ \text{on}\ M\setminus\bigcup B(x_i,4r_i),\]
    and $|\D_M\eta|\leq2|\D_M\bar\eta|$. Combined with \eqref{eq:area_in_ball}, and $|\D_M\eta_i|\leq Cr_i^{-1}$, we obtain
    \begin{equation}\label{eq:total_grad}
        \begin{aligned}
            \int_{\preg E}|\D_{\preg E}\eta|^2 &\leq \int_{\preg E}|\D_M\eta|^2\leq2\sum_i\int_{\preg E\cap B(x_i,4r_i)}|\D_M\eta_i|^2 \\
            &\leq C\sum_i r_i^{n-1}\cdot r_i^{-2}\leq C\delta^4.
        \end{aligned}
    \end{equation}
    
    For $\varphi\in C^\infty(M)$, let us consider a smooth family of sets $\{E_t\}_{t\in(-\varepsilon,\varepsilon)}$, such that $E_0=E$, the variational vector field $X_t$ along $\p E_t$ at $t=0$ is $\varphi\eta\nu$ (where $\nu$ denotes the outer unit normal of $\p E_t$), and $\nabla_{X_t}X_t = (\varphi\eta(\varphi\eta)_\nu)\nu$ at $t=0$. This family is well-defined since $\eta$ is supported inside $\preg E$. The variations of the area and the volume remain unchanged as in \cref{lem:Viscosity} (with each $\varphi$ replaced with $\varphi\eta$):
    \begin{equation}
        \frac{\mathrm{d} V}{\mathrm{d}t}(0)=\int_{\preg E}u^{\alpha}\varphi\eta,\qquad
        \frac{\mathrm{d}^2V}{\mathrm{d}t^2}(0)=\int_{\preg E}(H+\alpha u^{-1}u_\nu)u^\alpha\varphi^2\eta^2+u^\alpha\varphi\eta(\varphi\eta)_\nu,
    \end{equation}
    and
    \begin{equation}
        \frac{\mathrm{d}A}{\mathrm{d}t}(0)=\int_{\preg E}u^\alpha\varphi\eta\cdot u^{\gamma-\alpha}(H+\gamma u^{-1}u_\nu),
    \end{equation}
    and
    \[\begin{aligned}
        \frac{\mathrm{d}^2A}{\mathrm{d}t^2}(0) &= \int_{\preg E}\Big(-\Delta_{\preg E}(\varphi\eta)-\Ric(\nu,\nu)\varphi\eta-|\textrm{II}|^2\varphi\eta\Big)u^\gamma\varphi\eta \\
        &\qquad +\Big(-\gamma\varphi\eta u^{-2}u_\nu^2+\gamma u^{-1}\varphi\eta\D^2u(\nu,\nu)-\gamma u^{-1}\metric{\D_{\preg E}(u)}{\D_{\preg E}(\varphi\eta)}\Big)u^\gamma\varphi\eta \\
        &\qquad +\big(\gamma u^{\alpha-1}u_\nu\varphi^2\eta^2+u^{\alpha}\varphi\eta(\varphi\eta)_\nu +Hu^{\alpha}\varphi^2\eta^2\big) u^{\gamma-\alpha}(H+\gamma u^{-1}u_\nu).
    \end{aligned}\]
    By the same argument as in the proof of \cref{lem:Viscosity}, we have $u^{\gamma-\alpha}(H+\gamma u^{-1}u_\nu)=A'(v_0)$ on $\preg E$. Setting $\varphi=u^{-\gamma}$ in the variation, we have, running the same computations as in \cref{lem:Viscosity},    
    \[\begin{aligned}
        \frac{\mathrm{d}^2A}{\mathrm{d}t^2}(0) &\leq \int_{\preg E}
            u^{-\gamma}|\D_{\preg E}\eta|^2+u^{\alpha-2\gamma}\eta\eta_\nu A'(v_0) \\
            &\qquad -(n-1)\lambda u^{-\gamma}\eta^2+u^{-\gamma}\eta^2\Big[\frac{n-2}{n-1}X^2+\frac{2-n}{n-1}2\gamma XY+\Big(\frac{n-2}{n-1}\gamma^2-\gamma\Big)Y^2\Big].
    \end{aligned}\]
    Next we plug this into the chain rule \eqref{eqn:ChainRule} and use the facts $\alpha=\frac{2\gamma}{n-1}$, $0\leq\gamma\leq\frac{n-1}{n-2}$, and $u^\gamma\geq u^{2\alpha-3\gamma}$ (recall that we have normalized so that $\min(u)=1$, and $\gamma\geq\alpha$). Setting $Q:=\frac{\mathrm{d}V}{\mathrm{d}t}(0)=\int_{\preg E}u^{\alpha-\gamma}\eta$ we have:
    \[\begin{aligned}
        Q^2A''(v_0) &= \frac{\mathrm{d}^2A}{\mathrm{d}t^2}(0)-Q^{-1}\frac{\mathrm{d}A}{\mathrm{d}t}(0)\frac{\mathrm{d}^2V}{\mathrm{d}t^2}(0) \\
        &\leq \int_{\preg E}
            u^{-\gamma}|\D_{\preg E}\eta|^2+u^{\alpha-2\gamma}\eta\eta_\nu A'(v_0) \\
        &\qquad\qquad
            -(n-1)\lambda u^{-\gamma}\eta^2
            +u^{-\gamma}\eta^2\Big[\frac{n-2}{n-1}X^2+\frac{2-n}{n-1}2\gamma XY\Big] \\
        &\qquad
            -Q^{-1}\cdot QA'(v_0)
            \cdot\Big(\int_{\preg E}(X+\alpha Y-2\gamma Y)u^{\alpha-2\gamma}\eta^2+u^{\alpha-2\gamma}\eta\eta_\nu\Big) \\
        &\leq \int_{\preg E}u^{-\gamma}|\D_{\preg E}\eta|^2
            -(n-1)\lambda u^{2\alpha-3\gamma}\eta^2
            -\frac{A'(v_0)^2}{n-1}u^{2\alpha-3\gamma}\eta^2 \\
        &\leq C\delta^4-\Big(\frac{A'(v_0)^2}{n-1}+(n-1)\lambda\Big)\int_{\preg E}u^{2\alpha-3\gamma}\eta^2.
    \end{aligned}\]
    Hence, by using H\"older,
    \begin{equation}\label{eqn:AlmostViscosity}
        \begin{aligned}
            A''(v_0)\leq C\delta^4Q^{-2}-\Big(\int_{\preg E}u^\gamma\Big)^{-1}\Big(\frac{A'(v_0)^2}{n-1}+(n-1)\lambda\Big).
        \end{aligned}
    \end{equation}

    Notice that $I(v_0)=A(v_0)=\int_{\preg E} u^\gamma$, and 
    $$
    Q=\int_{\preg E} u^{\alpha-\gamma}\eta \geq \tilde C\int_{\preg E}\eta\geq\tilde C\H^{n-1}\big(\preg E\cap\{\eta=1\}\big).
    $$ 
    Observe that if $\delta\to 0$ we have that $\eta\to 1$ almost everywhere on $\partial^* E$. Thus for every $\varepsilon>0$ we can find $\delta_0>0$, so that the variation can be arranged such that $C\delta^4Q^{-2}\leq\varepsilon$ for all $\delta<\delta_0$. Thus taking \eqref{eqn:AlmostViscosity} into account the inequality
    \begin{equation}\label{eqn:NewAlmostDiff}
        II''\leq -\frac{(I')^2}{n-1}-(n-1)\lambda+\varepsilon I
    \end{equation}
    holds in the viscosity sense at $v_0$. Taking $\varepsilon\to 0$ the proof is concluded.
\end{proof}
\begin{remark}
    The proof offered above, when specialized to the case $\gamma=\alpha=0$, gives a short and concise proof of the main result of \cite{Bayle04}, i.e., \cite[Theorem 2.1]{Bayle04}.
\end{remark}

\begin{proof}[Proof of \cref{lem:SecondPointDiameter} ($n\geq 8$)]
Let $\Omega$ be a minimizer of \eqref{eqn:MuBubbleFunctional}. By the Riemannian analogue of \cite[Theorem 27.5]{MaggiBook}, and \cite[Theorem 28.1]{MaggiBook}, we have that the regular part of $\partial\Omega$ (denoted by $\preg \Omega$) is a smooth hypersurface, while the singular part of $\p \Omega$ (denoted by $\psing\Omega)$ has Hausdorff dimension at most $n-8$.

Let $\nu$ denote the outer unit normal at $\preg\Omega$. Similar to the argument above, for every $\delta\ll1$ there exists a function $\eta$ compactly supported in $\preg\Omega$, such that
\begin{equation}\label{eqn:EstimateOnGradientEta}
    \int_{\preg\Omega} |\nabla_{\preg \Omega}\eta|^2\leq C\delta^4,
\end{equation}
where $C$ does not depend on $\delta$. For the case here, the inequality \eqref{eq:area_in_ball} is proved by a direct comparison between $\Omega$ and $\Omega\cap B(x,r)$.

For $\varphi\in C^\infty(M)$, consider a family of sets $\{\Omega_t\}_{t\in(-\varepsilon,\varepsilon)}$, such that $\Omega_0=\Omega$ and the variational vector field along $\p \Omega_t$ is $\varphi\eta\nu$. The first variation becomes
\begin{equation}
    0=\frac{\mathrm{d}}{\mathrm{d}t}E(\Omega_t)\Big|_{t=0}=\int_{\preg\Omega}\big(H+\gamma u^{-1}u_\nu-hu^{\alpha-\gamma}\big)u^\gamma\varphi\eta.
\end{equation}
Since $\varphi$ is arbitrary, and the support of $\eta$ can be taken to exhaust $\preg\Omega$, we have $H=hu^{\alpha-\gamma}-\gamma u^{-1}u_\nu$ on $\preg\Omega$. Then, computing the second variation,
\[\begin{aligned}
    0 &\leq \frac {\mathrm{d}^2}{\mathrm{d}t^2}E(\Omega_t)\Big|_{t=0} \\
    &= \int_{\preg\Omega}\Big[-\Delta_{\preg\Omega}(\varphi\eta)-|\mathrm{II}|^2\varphi\eta-\Ric(\nu,\nu)\varphi\eta-\gamma u^{-2}u_\nu^2\varphi\eta \\
    &\qquad\qquad\quad +\gamma u^{-1}\varphi\eta\big(\Delta u-\Delta_{\preg\Omega}u-Hu_\nu\big)-\gamma u^{-1}\big\langle\D_{\preg\Omega}u,\D_{\preg\Omega}(\varphi\eta)\big\rangle \\
    &\qquad\qquad\quad -h_\nu u^{\alpha-\gamma}\varphi\eta+(\gamma-\alpha)hu^{\alpha-\gamma-1}u_\nu\varphi\eta\Big]u^\gamma\varphi\eta.
\end{aligned}\]
Setting $\varphi=u^{-\gamma}$, integrating by parts,
and discarding the negative term we get
\[\begin{aligned}
    0 &\leq \int_{\preg\Omega} u^{-\gamma}|\nabla_{\preg\Omega}\eta|^2-|\mathrm{II}|^2u^{-\gamma}\eta^2-\Ric(\nu,\nu)u^{-\gamma}\eta^2-\gamma u^{-2-\gamma}u_\nu^2\eta^2+\gamma u^{-1-\gamma}\eta^2\Delta u \\
    &\qquad -\gamma Hu^{-1-\gamma}u_\nu\eta^2-h_\nu u^{\alpha-2\gamma}\eta^2+(\gamma-\alpha)hu^{\alpha-2\gamma-1}u_\nu\eta^2 \\
    &\leq \int_{\preg\Omega}u^{-\gamma}|\nabla_{\preg\Omega}\eta|^2+u^{-\gamma}\eta^2\Big[-\frac{H^2}{n-1}-(n-1)\lambda-\gamma(u^{-1}u_\nu)^2 \\
    &\qquad -\gamma H(u^{-1}u_\nu)+|\D h|u^{\alpha-\gamma}+(\gamma-\alpha)hu^{\alpha-\gamma}(u^{-1}u_\nu)\Big].
\end{aligned}\]
Setting $X=hu^{\alpha-\gamma}$ and $Y=u^{-1}u_\nu$ (so $H=X-\gamma Y$), we have
\[\begin{aligned}
    0 &\leq \int_{\preg\Omega} u^{-\gamma}|\nabla_{\preg\Omega}\eta|^2+u^{-\gamma}\eta^2\Big[-\frac{X^2}{n-1}+\frac{2\gamma}{n-1}XY-\frac{\gamma^2}{n-1}Y^2-(n-1)\lambda \\
    &\qquad\qquad -\gamma Y^2-\gamma(X-\gamma Y)Y+|\D h|u^{\alpha-\gamma}+(\gamma-\alpha)XY\Big] \\
    &\leq \int_{\preg\Omega}u^{-\gamma}|\nabla_{\preg\Omega}\eta|^2+u^{-\gamma}\eta^2\Big[-\frac{X^2}{n-1}-(n-1)\lambda+|\D h|u^{\alpha-\gamma}\Big].
\end{aligned}\]
By \eqref{eqn:EstimateOnGradientEta}, if $h$ satisfies
\begin{equation}\label{eqn:ChangeInH}
    |\D h|u^{\alpha-\gamma}<(n-1)\lambda-C\delta^4+\frac{h^2u^{2\alpha-2\gamma}}{n-1}
\end{equation}
(with a different constant $C$, possibly depending also on $\preg \Omega$ but independent on $\delta$), then we obtain a contradiction. 

Given this, one finishes the proof by modifying the $n\leq7$ argument as follows. Let $0<\vartheta<(n-1)\lambda$. Arguing as in \eqref{eqn:EQNONH}$\,\sim\,$\eqref{eqn:TheSetO}, if
\begin{equation}\label{eq:aux2}
    \diam(M)>\frac{\pi}{\sqrt{\lambda-\vartheta/(n-1)}}\cdot\Big(\frac{\sup(u)}{\inf(u)}\Big)^{\frac{n-3}{n-1}\gamma}\cdot(1+\varepsilon)^2+2\varepsilon,
\end{equation}
then we can find a nontrivial stable $\mu$-bubble associated to some weight function $h$ satisfying $|\D h|\leq(n-1)\lambda-\vartheta+\frac{h^2u^{2\alpha-2\gamma}}{n-1}$. However, this causes a contradiction via \eqref{eqn:ChangeInH}, if we choose $\delta$ small enough. Thus \eqref{eq:aux2} cannot hold, and taking $\varepsilon,\vartheta\to0$ we obtain the result.
\end{proof}

\addtocontents{toc}{\protect\setcounter{tocdepth}{0}}

\printbibliography[heading=bibintoc,title={Bibliography}]

\addtocontents{toc}{\protect\setcounter{tocdepth}{2}}

\end{document}